\documentclass{article}[12pt]
\usepackage[title]{appendix}
\usepackage{mathrsfs}
\usepackage{mathtools}
\usepackage{color}
\usepackage{graphicx}
\usepackage{epstopdf}
\usepackage{stmaryrd}
 \usepackage{subfigure}
\usepackage{float}
\usepackage{amsmath}
    \usepackage{bm}
\usepackage{amsfonts,amssymb}
  \usepackage{dsfont}
\usepackage{amsfonts}
\usepackage{mathrsfs}
  \usepackage{pifont}
\numberwithin{equation}{section}
 \usepackage{amsthm}
 \newtheorem{theorem}{Theorem}[section]
\newtheorem{lemma}[theorem]{Lemma}

\newtheorem{remark}[theorem]{Remark}
\newtheorem{assumption}[theorem]{Assumption}
 \usepackage{amsmath}
 \usepackage{multirow}

\begin{document}
\title{
An immersed Crouzeix-Raviart finite element method in 2D and 3D based on discrete level set functions
}
\author{
Haifeng Ji\footnotemark[1]
}
\footnotetext[1]{School of Science, Nanjing University of Posts and Telecommunications, Nanjing, Jiangsu 210023, China  (hfji@njupt.edu.cn, hfji1988@foxmail.com)}

\date{}
\maketitle

\begin{abstract}
This paper is devoted to the construction and analysis of immersed finite element (IFE) methods in three dimensions. Different from the 2D case, the  points of intersection of the interface and the edges of a tetrahedron are usually not coplanar, which makes the extension of the original 2D IFE methods based on a piecewise  linear approximation of the interface to the 3D case not straightforward.
 We address this coplanarity issue by an approach where the interface is  approximated via discrete level set functions. 
 This approach is very convenient from a computational point of view since in many practical applications the exact interface is often unknown, and only a discrete level set function is available.
 As this approach has also not be considered in the 2D IFE methods, in this paper we present a unified framework for both 2D and 3D cases. We consider an IFE method based on the traditional Crouzeix-Raviart element using integral values on faces as degrees of freedom. The novelty of the proposed IFE is the unisolvence of basis functions on arbitrary triangles/tetrahedrons without any angle restrictions even for anisotropic interface problems, which is advantageous over the IFE using nodal values as degrees of freedom.  The optimal bounds for the IFE interpolation errors are proved on shape-regular triangulations.  For the IFE method, optimal a priori error and condition number estimates are derived with constants independent of the location of the interface with respect to the unfitted mesh. The extension to anisotropic interface problems with tensor coefficients is also discussed. Numerical examples supporting the theoretical results are provided.
\end{abstract}



\section{Introduction}\label{sec_introduce}
Let $\Omega$  be a bounded and convex polygonal/polyhedral  domain in $\mathbb{R}^N$, $N=2$ or $3$, and 
the interface $\Gamma$ be a  $C^2$ compact hypersurface without boundary which is embedded in $\Omega$ and divides $\Omega$ into two disjoint subdomains $\Omega^+$ and $\Omega^-$.  Without loss of generality, we assume that $\Omega^-$ lies inside $\Omega$ strictly, i.e.,  $\Gamma=\partial\Omega^-$.
Consider  the following second-order elliptic interface problem with variable coefficients
\begin{align}
-\nabla\cdot(\beta(\mathbf{x})\nabla u)&=f  ~~~~\qquad\mbox{in } \Omega^+\cup \Omega^-,\label{p1.1}\\
[u]_{\Gamma}&=0~~~~ \qquad\mbox{on } \Gamma,\label{p1.2}\\
[\beta\nabla u\cdot \mathbf{n}]_{\Gamma}&=0~~~~ \qquad\mbox{on } \Gamma,\label{p1.3}\\
u&=0 ~~~~\qquad\mbox{on } \partial\Omega,\label{p1.4}
\end{align}
where $f\in L^2(\Omega)$,  $\mathbf{n}(\mathbf{x})$ denotes the unit normal vector to $\Gamma$ at point $\mathbf{x}\in\Gamma$ pointing from $\Omega^-$ to $\Omega^+$, $[v]_\Gamma$  denotes the jump of a function $v$ across the interface, i.e.,
\begin{equation*}
[v]_{\Gamma}:=v^+|_\Gamma-v^-|_\Gamma~~ \mbox{on}~\Gamma ~~\mbox{with} ~~v^\pm=v|_{\Omega^\pm},
\end{equation*}
and the coefficient $\beta(\mathbf{x})$  can be discontinuous across the interface $\Gamma$ and is assumed to be piecewise smooth such that
\begin{equation}\label{p1.5}
\beta(\mathbf{x})|_{\Omega^\pm}=\beta^\pm(\mathbf{x}) ~~\mbox{ with }~~\beta^\pm(\mathbf{x})\in C^1(\overline{\Omega^\pm}).
\end{equation}
We also assume  there exist positive constants $\beta^\pm_{m}$ and $\beta^\pm_{M}$  such that $\beta^\pm_{m}\leq \beta^\pm(\mathbf{x})\leq\beta^\pm_{M}$. 
The anisotropic interface problem, i.e., the coefficient $\beta(\mathbf{x})$ is replaced by a discontinuous tensor-valued function $\mathbb{B}(\mathbf{x})$, will be discussed in Section~\ref{sec_exten}.
 
Interface problems appear in many engineering and physical applications involving multiple materials and interfaces. The main challenge is that  the solutions of interface problems are not smooth across interfaces due to interface conditions and discontinuous coefficients. It is well known that  finite element methods (FEMs) can be used to solve interface problems with optimal accuracy based on body-fitted and shape-regular meshes (see, e.g., \cite{0xu,bramble1996finite,chen1998finite}). However, it is not trivial and time-consuming to generate such a shape-regular mesh that fits complex or moving interfaces especially in 3D. So, FEMs based on unfitted meshes, which are completely independent of the interface, have become highly attractive for interface problems. 
There are many FEMs using unfitted meshes (called unfitted mesh methods) in the literature, for example, the unfitted Nitsche's method \cite{hansbo2002unfitted,2016High,burman2015cutfem}, the extended FEM \cite{fries2010extended}, the multiscale FEM \cite{multi_CHU2010}, the FEM for high-contrast problems \cite{GuzmanJSC2017}, the immersed virtual element method (IVEM) \cite{Cao2021Immersed},  and the immersed finite element (IFE) method \cite{li1998immersed,Li2003new,taolin2015siam,Slimane2020,Chang2011An}.  

We are interested in the IFE method which is distinguished from other unfitted mesh methods in the fact that the degrees of freedom are the same as that of standard FEMs and the IFE space is isomorphic to the standard finite element space. This feature is advantageous when dealing with moving interface problems \cite{guo2021SIAM} and interface inverse problems \cite{guo2019fixed}.  The basic idea of the IFE method is fairly simple: modify the basis functions of standard FEMs on interface elements according to the jump conditions to capture the jump behaviors of the exact solution. Actually, this idea can be traced back to the fundamental work of Babu\v{s}ka et al. in \cite{babuvska1994special} where special basis functions are obtained by solving local problems to capture the behaviors of exact solutions. We note that the local problems are also used in the virtual element method (VEM) with variable coefficients. As pointed out in \cite{Cao2021Immersed}, for 1D problems with a piecewise constant coefficient $\beta$, the IFE space in \cite{li1998immersed}, the finite element space in \cite{babuvska1994special}, and the virtual element space are exactly identical due to the trivial 1D geometry, but they are distinguished in higher dimensions because of the more complicated geometry.
%the geometry is more complicated. 
For the existing 2D IFE methods (see, e.g., \cite{Li2003new,taolin2015siam,2021ji_IFE}), the interface inside an interface element is approximated by a straight line connecting the intersection points of the interface and the edges of the element, and a piecewise linear function is used as the IFE basis function so that the interface conditions can be satisfied on the straight line.  The optimal approximation capabilities of the IFE spaces and the analysis of the related IFE methods have been presented in \cite{taolin2015siam,wang2020rigorous,GuoIMA2019,2021ji_IFE}.  

However, for real 3D problems, the IFE methods and the corresponding theoretical analysis are relatively few; see \cite{kafafy2005three,hou2013weak,2021ji_IFE} for linear IFE methods on tetrahedral meshes,  \cite{vallaghe2010trilinear,guo2020immersed,guo2021solving} for trilinear IFE methods on cuboidal meshes,  and \cite{han2021pife} for some applications. 
%We note that only the trilinear IFE method in \cite{guo2020immersed,guo2021solving} has the rigorous theoretical analysis.  
Different from the 2D case,  the  points of intersection of the interface and the edges of an interface element are usually not coplanar. So, it is impossible to make a piecewise linear function continuous at these intersection points.  In the methods proposed in \cite{kafafy2005three,guo2020immersed,guo2021solving}, the authors carefully choose three of intersection points to determine a plane approximating the exact interface and construct IFE functions based on the interface conditions defined on the plane. Another approach proposed in \cite{hou2013weak} is to use all the intersection points, leading to  to an over-determined system of equations. The IFE functions are then obtained by the least squares method. To our best knowledge, there is no theoretical results for this approach.
 
%In this paper we aim to extend our previous work on 2D nonconforming IFE methods in \cite{2021ji_nonconform}  to three dimensions.  We address the coplanarity issue by using a continuous linear approximation of the interface which can be obtained by the zero level set of  the linear interpolant  of the signed distance function to the interface.  
In this paper we address the coplanarity issue by using a continuous linear approximation of the interface which can be obtained by the zero level set of  the linear interpolant  of the signed distance function to the interface.  
Since this approach has also not be discussed in 2D, we present a unified framework for both 2D and 3D cases.
%
%Note that this approach has also not be discussed in 2D,  thus we present a unified framework for both 2D and 3D IFE methods %based on triangular/tetrahedral meshes and the Crouzeix-Raviart finite element.  
%
%
Different from the method in \cite{guo2020immersed,guo2021solving,2021ji_IFE}, we use the discrete interface in both the IFE space and the IFE method, which is very convenient from a computational point of view.  Note that the approximation of the interface in \cite{guo2020immersed,guo2021solving,2021ji_IFE}  is only used for providing connection conditions for the piecewise polynomial basis functions, and the IFE functions and methods are defined according to the exact interface since the approximate  interface on interface elements cannot form a continuous surface. 
We develop and analyze an IFE method based on the conventional Crouzeix-Raviart finite element  using integral values as degrees of freedom \cite{crouzeix1973conforming} on triangular/tetrahedral meshes, which is an extension of our previous work on 2D nonconforming IFE methods in \cite{2021ji_nonconform}.
We prove that the IFE basis functions  are unisolvent on arbitrary triangles/tetrahedrons without any angle restrictions.
We note that if the values on vertices are used as degrees of freedom, the unisolvence relies on some mesh assumptions; see for example the “no-obtuse-angle” condition introduced in \cite{2021ji_IFE} for both 2D and 3D problems.  
We prove the optimal approximation capabilities of the proposed IFE space under the assumption that the triangulation is shape-regular.  
The proof is based on the method proposed in \cite{2021ji_IFE} where tangential gradients and  their corresponding extensions are defined via the signed distance function near the interface.
The approximation of the interface via discrete level set functions brings new difficulties because there may be  no intersection points between the exact interface and the discrete interface on an interface element.
For the proposed IFE method, by establishing the trace inequality and the inverse inequality for IFE functions, we derive the optimal a priori error and condition number estimates with constants independent of the location of the interface with respect to the unfitted mesh. We also provide some numerical examples to  validate the theoretical results.

Another contribution of this paper is the finding that for the case of tensor-valued coefficients, the IFE basis functions based on integral-value degrees of freedom are also unisolvent on arbitrary triangles/tetrahedrons, and consequently the theoretical analysis proposed in this paper can be readily extended to this case.  It should be noted that the IFE basis functions based on nodal-value degrees of freedom may not exist for this case even in 2D (see \cite{An2014A}).

The remainder of the paper is organized as follows. In Section~\ref{sec_pre},  some necessary notations and preliminary results are presented. In Section~\ref{sec_IFE}, we first introduce unfitted meshes, the discrete interface, and the assumptions and notations, and then present the immersed Crouzeix-Raviart finite elements.  Section~\ref{sec_pro} is devoted to the properties of the proposed IFEs including the unisolvence of the IFE basis functions and the optimal approximation capabilities of the IFE space. In Section~\ref{sec_IFEM}, the IFE method and the corresponding analysis are presented.
In Section~\ref{sec_exten}, the extension to anisotropic interface problems is discussed. Numerical examples are given in Section~\ref{lem_num}. Finally, some conclusions are drawn in Section~\ref{sec_con}.

\section{Preliminaries}\label{sec_pre}
Let $k\geq 0$ be an integer and $1\leq p\leq \infty$ be a real number.  We adopt the standard notation $W^k_p(D)$ for Sobolev spaces on a domain $D$ with the norm $\|\cdot\|_{W^k_p(D)}$ and the seminorm $|\cdot|_{W^k_p(D)}$. Specially, $W^k_2(D)$ is denoted by $H^{k}(D)$ with the norm  $\|\cdot\|_{H^{k}(D)}$ and the seminorm $|\cdot|_{H^{k}(D)}$. As usual $H_0^1(D)=\{v\in H^1(D) : v=0 \mbox{ on }\partial D\}$.
For any subdomain $D\subset \mathbb{R}^N$, we define subdomains  $D^\pm:=D\cap \Omega^\pm$ and a broken Sobolev space via
\begin{equation*}
H^k(\cup D^\pm)=\{v\in L^2(D) : v|_{D^\pm} \in H^k(D^\pm)\},
\end{equation*}
which is equipped with the norm $\|\cdot\|_{H^k(\cup D^\pm)}$ and the  semi-norm $|\cdot|_{H^k(\cup D^\pm)}$ satisfying
$$
\|\cdot\|^2_{H^k(\cup D^\pm)}=\|\cdot\|^2_{H^k(D^+)}+\|\cdot\|^2_{H^k(D^-)}, \quad|\cdot|^2_{H^k(\cup D^\pm)}=|\cdot|^2_{H^k(D^+)}+|\cdot|^2_{H^k(D^-)}.
$$
For the elliptic interface problems, we introduce a subspace of $H^2(\cup D^\pm)$, 
\begin{equation}\label{def_H2}
\widetilde{H}^2(D)=\{v\in H^2(\cup D^\pm) :  [v]_{\Gamma\cap D}=0,~ [\beta\nabla v\cdot \mathbf{n}]_{\Gamma\cap D}=0\}.
\end{equation}
Obviously, $\widetilde{H}^2(D)\subset H^1(D)$.  Under the setting introduced in Section~\ref{sec_introduce}, it can be shown that (see \cite{2012Uniform}) the interface problem (\ref{p1.1})-(\ref{p1.5}) has a unique solution $u\in \widetilde{H}^2(\Omega)\cap H_0^1(\Omega)$ satisfying the following a priori estimate
\begin{equation}\label{regular}
\|u\|_{H^2(\cup \Omega^\pm)}\leq C\|f\|_{L^2(\Omega)}.
\end{equation}
In our analysis, we will frequently use the the signed distance function
\begin{equation*}
d(\mathbf{x})=\left\{
\begin{aligned}
&\mbox{dist}(\mathbf{x},\Gamma)\quad&& \mbox{ if } \mathbf{x}\in\overline{\Omega^+},\\
&-\mbox{dist}(\mathbf{x},\Gamma)\quad&& \mbox{ if } \mathbf{x}\in\Omega^-.
\end{aligned}\right.
\end{equation*}
Define the $\delta$-neighborhood of $\Gamma$ by
\begin{equation*}
U(\Gamma,\delta)=\{x\in\mathbb{R}^N: \mbox{dist}(\mathbf{x},\Gamma)< \delta\}.
\end{equation*}
It is well known that $d(\mathbf{x})$ is globally Lipschitz-continuous, and for $\Gamma\in C^2$, there exists $\delta_{0}>0$ such that $d(\mathbf{x})\in C^2\left(U(\Gamma,\delta_0)\right)$ (see \cite{foote1984regularity}) and the closest point mapping $\mathbf{p}: U(\Gamma,\delta_0)\rightarrow \Gamma$ maps every $\mathbf{x}$ to precisely one point at $\Gamma$. In other words, every point $\mathbf{x}\in U(\Gamma,\delta_0)$ can be uniquely written as 
\begin{equation*}
\mathbf{x}=\mathbf{p}(\mathbf{x})+d(\mathbf{x})\mathbf{n}(\mathbf{p}(\mathbf{x})).
\end{equation*}
The existence of $\delta_0$ is a standard result in differential geometry. For example, for $N=3$, we require that $\delta_0< \left(\max_{i=1,2}\|\kappa_i\|_{L^\infty(\Gamma)}\right)^{-1}$, where $\kappa_1$ and $\kappa_2$ are the principal curvatures of $\Gamma$ (see (2.2.4) in \cite{demlow2007adaptive}).

Define $U^\pm(\Gamma,\delta)=U(\Gamma,\delta)\cap \Omega^\pm$. We recall the following fundamental inequality that will be useful in our analysis.
\begin{lemma}\label{lem_strip}
For all $\delta\in (0,\delta_0]$, there is a constant $C$ depending only on $\Gamma$ such  that  
\begin{equation}\label{strip}
\|v\|^2_{L^2(U^\pm(\Gamma,\delta))}\leq C\left(\delta \|v\|^2_{L^2(\Gamma)}+\delta^2  \|\nabla v\|^2_{L^2(U^\pm(\Gamma,\delta))}\right)\qquad\forall  v\in H^1(U^\pm(\Gamma,\delta)).
\end{equation}
\end{lemma}
\begin{proof}
See (A.8)-(A.10) in \cite{burman2018Acut}. 
\end{proof}
%The above lemma can be viewed as an extension of Lemma 2 in \cite{James1994A} where $\delta=O(h^2)$ and $N=2$.
\begin{remark}
If $v\in H^1(U^\pm(\Gamma,\delta_0))$, then applying the global trace inequality to $\|v\|_{L^2(\Gamma)}$ on $U^\pm(\Gamma,\delta_0)$,  the inequality (\ref{strip}) becomes 
\begin{equation}\label{delta_est_yaun}
\|v\|^2_{L^2(U^\pm(\Gamma,\delta))}\leq C\delta  \| v\|^2_{H^1(U^\pm(\Gamma,\delta_0))},
\end{equation}
which was proved in \cite{Li2010Optimal,elliott2013finite}. Furthermore, if $v|_\Gamma=0$, the inequality (\ref{strip}) implies 
\begin{equation}\label{delta_est_yaun2}
\|v\|^2_{L^2(U^\pm(\Gamma,\delta))}\leq C\delta^2  \|\nabla v\|^2_{L^2(U^\pm(\Gamma,\delta))},
\end{equation}
which was also proved in \cite{Li2010Optimal}. We note that the constant $C$ depends on $\delta_0$, but not on $\delta$.
\end{remark}

\section{Immersed finite elements}\label{sec_IFE}
In this section we first introduce unfitted meshes, the discrete interface,  and the assumptions and notations. Then we present the immersed Crouzeix-Raviart finite element in 2D and 3D.
\subsection{Unfitted meshes}
 Let $ \{\mathcal{T}_h\}_{h>0}$ be a family of simplicial triangulations of the domain $\Omega$, generated independently of the interface $\Gamma$. For an element $T\in\mathcal{T}_h$ (a triangle for $N=2$ and  a tetrahedron for $N=3$),  $h_T$ denotes its diameter, and for a mesh $\mathcal{T}_h$, the index $h$ refers to the maximal diameter of all elements in $\mathcal{T}_h$, i.e., $h=\max_{T\in\mathcal{T}_h}h_T$. We assume that $\mathcal{T}_h$ is shape-regular, i.e., for every $T\in\mathcal{T}_h$, there exists a positive constant $\varrho$ such that  $ h_T\leq \varrho r_T$ where $r_T$ is the radius of the largest ball inscribed in $T$.  
In this paper, face means edge/face in two/three dimensions.
Denote $\mathcal{F}_h$  as the set of faces of the triangulation $\mathcal{T}_h$, and let $\mathcal{F}^\circ_h$ and $\mathcal{F}^b_h$  be the sets of interior faces and boundary faces.  
We adopt the convention that  elements  and  faces are open sets. 
Then the sets of interface elements and interface faces are defined as
\begin{equation*}
\mathcal{T}_h^\Gamma =\{T\in\mathcal{T}_h :  T\cap \Gamma\not = \emptyset\} \quad\mbox{ and }\quad \mathcal{F}_h^\Gamma=\{F\in \mathcal{F}_h : F \cap \Gamma\not = \emptyset\}.
\end{equation*}
The sets of non-interface elements and non-interface faces are $\mathcal{T}^{non}_h=\mathcal{T}_h\backslash\mathcal{T}_h^{\Gamma}$ and $\mathcal{F}^{non}_h=\mathcal{F}_h\backslash\mathcal{F}_h^{\Gamma}$, respectively.

Define $h_\Gamma=\max_{T\in\mathcal{T}_h^\Gamma}h_T.$  Our method and analysis will be valid when $h_\Gamma$ is sufficiently small so that the interface is  resolved by the unfitted mesh in the sense that  the following assumptions are satisfied.
\begin{assumption}\label{assum_2}
We can always refine the mesh near the interface to satisfy:
\begin{itemize}
  \item $h_\Gamma<\delta_0$ so that $\overline{T}\subset U(\Gamma,\delta_0)$ for all $T\in\mathcal{T}_h^\Gamma$.
  \item For any triangle belonging  to $\mathcal{T}^\Gamma_h$ for $N=2$, or belonging to $\mathcal{F}^\Gamma_h$ for $N=3$,   the interface $\Gamma$  must intersect the boundary of the triangle at two points,  and these two points cannot be on the same edge (including two endpoints) of the triangle.
\end{itemize}
\end{assumption}

 \begin{figure} [htbp]
\centering
\subfigure[2D]{\label{fig_inte_2d}%% label for first subfigure
\includegraphics[width=0.25\textwidth]{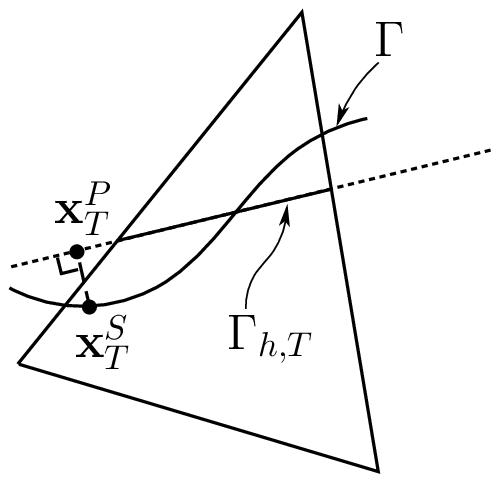}}\quad~
\subfigure[3D: Type I ]{\label{fig_inte_3d1}%% label for second subfigure
\includegraphics[width=0.25\textwidth]{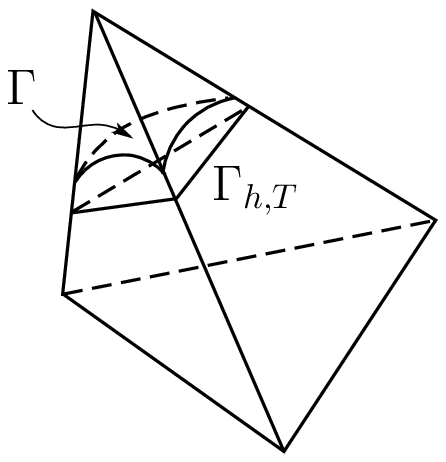}}\quad~
\subfigure[3D: Type II]{\label{fig_inte_3d2}%% label for second subfigure
\includegraphics[width=0.23\textwidth]{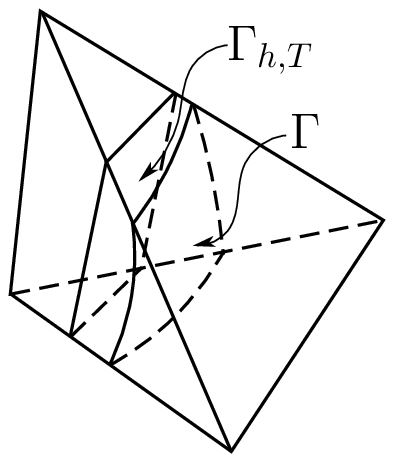}}\quad~
 \caption{Intersection topologies of interface elements\label{fig_inter}} %% label for entire figure
\end{figure}

Based on the above assumption, we now investigate the possible intersection topologies of interface elements. For $N=2$, there is only one type of the interface elements (see Figure~\ref{fig_inte_2d}).
However, for $N=3$,  we have two types of the interface elements as shown by Type I (Three-edge cut) in Figure~\ref{fig_inte_3d1} and by Type II (Four-edge cut) in Figure~\ref{fig_inte_3d2}. 

 Note that the case that  the interface intersects an  interface element at  some vertices is also taken into account in this classification by viewing it as the limit situation of one of these types. We also note that the case that some faces are part of the interface or all vertices of some faces are on the interface can be easily treated as body-fitted meshes, so we do not consider this case in this paper for  simplicity of presentation.

\subsection{Discretization of the interface}
Let  us  denote the discrete interface by $\Gamma_h$, which partitions $\Omega$ into two subdomains $\Omega_h^+$ and $\Omega_h^-$ with $\partial \Omega_h^-=\Gamma_h$.  Define $\Gamma_{h,T}:=\Gamma_h \cap T$ and $\Gamma_{T}:=\Gamma \cap T$ for all $T\in\mathcal{T}_h^\Gamma$. We make the following abstract assumptions.
\begin{assumption}\label{ass_Gamma_h}
The discrete interface $\Gamma_h$ is chosen such that

\begin{itemize}
  \item  The discrete interface $\Gamma_h$ is  $C^0$-smooth and is composed of $\Gamma_{h,T}\subset \mathbb{R}^{N-1}$ for all  interface element $T\in\mathcal{T}_h^\Gamma$, i.e., $\Gamma_{h,T}$ is a line segment for $N=2$ and a planar segment for $N=3$ (see, e.g.,  Figure~\ref{fig_inter}). 
  \item The closest point mapping  $\mathbf{p}|_{\Gamma_h}: \Gamma_h\rightarrow \Gamma $ is a bijection. 
  \item There is a positive constant $C$ independent of $h$ and the interface location relative to the mesh such that for all $T\in\mathcal{T}_h^\Gamma$,
\begin{align}
\|d(\mathbf{x})\|_{L^\infty(\Gamma_{h,T})}\leq Ch_T^2,\label{ass_Gamma_h_1}\\
\|\mbox{dist}(\mathbf{x},\Gamma_{h,T}^{ext})\|_{L^\infty(\Gamma_T)}\leq Ch_T^2,\label{ass_Gamma_h_2}\\
\|\mathbf{n}-\mathbf{n}_{h}\|_{L^\infty(\Gamma_{T})}\leq Ch_T,\label{ass_Gamma_h_3}
\end{align}
where $\Gamma^{ext}_{h,T}$ is a $N-1$-dimensional hyperplane containing $\Gamma_{h,T}$ and $\mathbf{n}_{h}$ is a piecewise constant vector defined on interface elements with $\mathbf{n}_{h}|_T$ being the unit vector perpendicular to $\Gamma_{h,T}$  pointing from $\Omega^-_h$ to $\Omega^+_h$.
\end{itemize}
\end{assumption}

 \begin{figure} [htbp]
\centering
\includegraphics[width=0.5\textwidth]{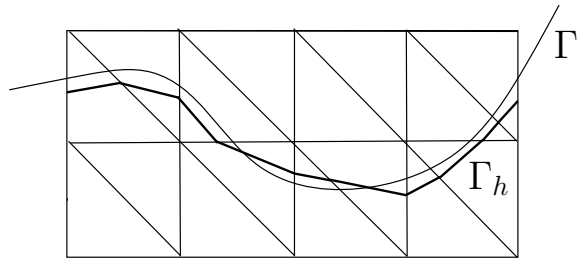}
 \caption{An illustration of $\Gamma_h$ for the 2D case\label{fig_gamma}} %% label for entire figure
\end{figure}

We emphasize that the hyperplane $\Gamma^{ext}_{h,T}$ plays an important role in the analysis of IFE methods. In the construction of IFE spaces, one often uses $v^+-v^-=0$ on $\Gamma_{h,T}$ to enforce the continuity, where $v^\pm$ are linear functions. This implies $v^+-v^-=0$ on $\Gamma^{ext}_{h,T}$. The latter is more beneficial for analysis. See Remarks \ref{remark_th} and \ref{remark_xtxp} for details.

In Figure~\ref{fig_gamma}, we illustrate an example of this discrete interface $\Gamma_h$  for the two-dimensional case.  Here we do not investigate whether (\ref{ass_Gamma_h_1}) and (\ref{ass_Gamma_h_2}) are independent or not because they can be easily verified in practical applications. Under these assumptions, we  now derive some relations that will be useful in our analysis.
Using the signed distance function $d(\mathbf{x})$,  we have $\mathbf{n}(\mathbf{x})=\nabla d(\mathbf{x})$, which is well-defined in $U(\Gamma,\delta_0)$. As we assume that $\Gamma\in C^2$, it holds $d(\mathbf{x})\in C^2\left(U(\Gamma,\delta_0)\right)$ (see \cite{foote1984regularity}), and hence  $\mathbf{n}(\mathbf{x})\in C^1\left( U(\Gamma,\delta_0) \right)^N$. Therefore, the inequality (\ref{ass_Gamma_h_3}) in Assumption~\ref{ass_Gamma_h} implies
\begin{equation}\label{n_h_esti}
\begin{aligned}
 \|\mathbf{n}-\mathbf{n}_{h}\|_{L^\infty(T)}&=|\mathbf{n}(\mathbf{x})-\mathbf{n}_{h}(\mathbf{x}_\Gamma)|\\
 &\leq |\mathbf{n}(\mathbf{x})-\mathbf{n}(\mathbf{x}_\Gamma)|+|\mathbf{n}(\mathbf{x}_\Gamma)-\mathbf{n}_{h}(\mathbf{x}_\Gamma)|\\
&\leq C|\mathbf{x}-\mathbf{x}_\Gamma |+  \|\mathbf{n}-\mathbf{n}_{h}\|_{L^\infty(\Gamma_{T})}\\
&\leq Ch_T,
 \end{aligned}
\end{equation}
where  $\mathbf{x}\in \overline{T}$, $\mathbf{x}_\Gamma\in \overline{\Gamma_T}$, and $|\cdot|$ stands for the 2-norm of a vector.  In addition, the inequality (\ref{ass_Gamma_h_1}) in Assumption~\ref{ass_Gamma_h}  implies that there exists a constant $C_\Gamma$ independent of $h$ and the interface location relative to the mesh such that
\begin{equation*}
%\label{gamma_hsubset}
\Gamma_h\subset U(\Gamma,C_\Gamma h_\Gamma^2).
\end{equation*} 
The mismatch region caused by the discretization of the interface is defined by $\Omega^\triangle:=(\Omega_h^-\cap\Omega^+)\cup(\Omega_h^+\cap\Omega^-).$ Also define  $T_h^\pm:=T\cap \Omega_h^\pm$ and $T^\triangle:=(T_h^-\cap T^+)\cup(T_h^+\cap T^-)$ for all $T\in\mathcal{T}_h^\Gamma$.
Obviously, we have
\begin{equation}\label{tdelta_U}
\Omega^\triangle=\bigcup_{T\in\mathcal{T}_h^\Gamma} T^\triangle~\mbox{ and }~ \Omega^\triangle \subset U(\Gamma,C_\Gamma h_\Gamma^2).
\end{equation}
The inequality (\ref{ass_Gamma_h_2}) in Assumption~\ref{ass_Gamma_h}  is used to derive (\ref{tpts}), which is useful in the analysis (see Remark~\ref{remark_xtxp}).

Now  we give an example of the discrete interface $\Gamma_h$ that fulfills  Assumption~\ref{ass_Gamma_h}. Let $I_h$ be the piecewise linear nodal interpolation operator associated with $\mathcal{T}_h$. The discrete interface can be chosen as the zero level set of the Lagrange interpolant of $d(\mathbf{x})$, i.e.,
\begin{equation*}
\Gamma_h:=\{\mathbf{x}\in\mathbb{R}^N : I_hd(\mathbf{x})=0\}.
\end{equation*}
This choice of $\Gamma_h$ is often used in the CutFEM for solving PDEs on surfaces (see, e.g., \cite{burman2017cut}).  The first two properties in Assumption~\ref{ass_Gamma_h} are obviously satisfied. It suffices to verify (\ref{ass_Gamma_h_1})-(\ref{ass_Gamma_h_3}). Since $d(\mathbf{x})\in C^2\left(U(\Gamma,\delta_0)\right)$, we have for all $T\in\mathcal{T}_h^\Gamma$ that
\begin{equation}\label{Ihdpro}
\|d-I_hd\|_{L^\infty(T)}+h_T\|\nabla d-\nabla I_hd\|_{L^\infty(T)}\leq Ch_T^2,
\end{equation}
which together with the facts 
$$\mathbf{n}=\nabla d, ~~|\nabla d|=1, ~~\mathbf{n}_h=|\nabla I_hd|^{-1}\nabla I_hd, ~~\left||\nabla I_hd|^{-1}I_hd\right|=\mbox{dist}(\mathbf{x},\Gamma_{h,T}^{ext})\leq Ch_T~~\forall \mathbf{x}\in T,$$ 
leads to
\begin{equation*}
\begin{aligned}
&\left||\nabla I_hd|-1\right|\leq | \nabla I_hd -\nabla d|\leq Ch_T,\\
&\|d\|_{L^\infty(\Gamma_{h,T})}=\|d-I_hd\|_{L^\infty(\Gamma_{h,T})}\leq \|d-I_hd\|_{L^\infty(T)}\leq Ch_T^2,\\
&\|\mathbf{n}-\mathbf{n}_h\|_{L^\infty(\Gamma_T)}=\left\|\nabla d - |\nabla I_hd|^{-1}\nabla I_hd\right\|_{L^\infty(\Gamma_T)}\\
&~~~\qquad\qquad\qquad\leq \left\|\nabla d - \nabla I_hd\right\|_{L^\infty(\Gamma_T)}+\left\|\nabla I_hd - |\nabla I_hd|^{-1}\nabla I_hd\right\|_{L^\infty(\Gamma_T)}\\
&~~~\qquad\qquad\qquad= \left\|\nabla d - \nabla I_hd\right\|_{L^\infty(\Gamma_T)}+\left||\nabla I_hd|-1\right|\\
&~~~\qquad\qquad\qquad\leq Ch_T,\\
& \|\mbox{dist}(\mathbf{x},\Gamma_{h,T}^{ext})\|_{L^\infty(\Gamma_T)}= \left\||\nabla I_hd|^{-1}I_hd\right\|_{L^\infty(\Gamma_T)}\\
&~~~~~~~~~~\qquad\qquad\qquad=\left\||\nabla I_hd|^{-1}I_hd-d\right\|_{L^\infty(\Gamma_T)}\\
&~~~~~~~~~~\qquad\qquad\qquad\leq \left\||\nabla I_hd|^{-1}I_hd-I_hd\right\|_{L^\infty(\Gamma_T)}+\left\|I_hd-d\right\|_{L^\infty(\Gamma_T)}\\
&~~~~~~~~~~\qquad\qquad\qquad= \left|1-|\nabla I_hd|\right|  \left\|   |\nabla I_hd|^{-1} I_hd\right\|_{L^\infty(\Gamma_T)}+\left\|I_hd-d\right\|_{L^\infty(\Gamma_T)}\\
&~~~~~~~~~~\qquad\qquad\qquad\leq Ch_T^2.
\end{aligned}
\end{equation*}
This completes the verification of (\ref{ass_Gamma_h_1})-(\ref{ass_Gamma_h_3}).

In many practical applications, the exact interface is unknown, and only a discrete level set function $d_h(\mathbf{x})$ is available which is often obtained by solving the related PDEs for the interface together with some redistancing procedures (see, e.g., \cite{elsey2014fast}).  The discrete interface is then chosen as the zero level set of $d_h(\mathbf{x})$.  For this point of view,  the IFE method developed in this paper is particularly well suited.

We also note that $d_h(\mathbf{x})$ and the corresponding $\Gamma_h$ can also be obtained if the exact interface is given by a parametric representation because there exist algorithms to compute the closet point projection based on the parametric representation of the exact interface (see, e.g., \cite{Rangarajan2011}).

\subsection{Extensions and notation}
For any $\delta\in (0,\delta_0]$, define $\Omega^\pm_{\delta}:=\Omega^\pm\cup U(\Gamma,\delta)$.
It is well known that there exist  extension operators $E^\pm$: $H^m(\Omega^\pm)\rightarrow H^m(\Omega^\pm_{\delta_0})$ for any $m\geq0$ such that
\begin{equation}\label{extension}
(E^\pm v^\pm)|_{\Omega^\pm}=v^\pm~\mbox{and}~\|E^\pm v\|_{H^m(\Omega^\pm_{\delta_0})}\leq C\|v^\pm\|_{H^m(\Omega^\pm)}~\mbox{for all}~ v^\pm\in H^m(\Omega^\pm), 
\end{equation}
where the constant $C$ depends on $\Omega^\pm$ (see \cite{Gilbargbook}). For brevity we shall use the notation $v_E^+$  and $v_E^-$ for the extended functions $E^+ v^+$ and $E^- v^-$, i.e., $v_E^\pm:=E^\pm v^\pm$. 

For the discontinuous coefficients, since $\beta^\pm(\mathbf{x})\in C^1(\overline{\Omega^\pm})$, we can further assume that the extensions also satisfy 
\begin{equation}\label{beta_ext1}
\beta_E^\pm(\mathbf{x})\in C^1(\overline{\Omega^\pm_{\delta_0}})\quad\mbox{ and }\quad\tilde{\beta}^{\pm}_{m}\leq \beta_E^\pm(\mathbf{x})\leq\tilde{\beta}^\pm_{M}\quad\forall \mathbf{x}\in \overline{\Omega^\pm_{\delta_0}},
\end{equation}
where the constants $\tilde{\beta}^\pm_{m}$ and $\tilde{\beta}^\pm_{M}$ are positive and depend on $\delta_0$ and $\beta^\pm$. Thus, there exists a constant $C_\beta\geq 0$ depending on $\beta_E^\pm$ such that
\begin{equation}\label{beta_ext2}
\|\nabla \beta^\pm_E\|_{L^\infty(\Omega^\pm_{\delta_0})}\leq C_{\beta}.
\end{equation}
Note that if $\beta$ is a piecewise constant, the constant $C_\beta=0$.

We now consider the extension of polynomials.
Let $\mathbb{P}_k(D)$ be the set of all polynomials of degree less than or equal to $k$ on the domain $D$. 
Given a function $v\in L^2(T)$ with $v|_{T_h^\pm}\in \mathbb{P}_k(T_h^\pm)$, with a small ambiguity of notation,  we use  $v^\pm$  to represent the polynomial extension of $v|_{T_h^\pm}$, i.e.,
\begin{equation*}
v^\pm\in  \mathbb{P}_k(T)~~\mbox{ and } ~~ v^\pm|_{T_h^\pm}=v|_{T_h^\pm}.
\end{equation*}
We note that the superscripts $+$ and $-$ are also used for the restrictions of a function $v\in L^2(\Omega)$ on $\Omega^\pm$,  i.e., $v^\pm:=v|_{\Omega^\pm}$.  This abuse of notation will not cause any confusion in the analysis but simplifies the notation greatly. The reason is that we often use the extensions $v_E^+$ and $v_E^-$ when $v^\pm$ means $v|_{\Omega^\pm}$.

Given a bounded domain $D$, for any $v^\pm \in L^2(D)$, we define
\begin{equation*}
[\![v^\pm]\!](\mathbf{x}):=v^+(\mathbf{x})-v^-(\mathbf{x}) \quad\forall \mathbf{x}\in D.
\end{equation*}
Therefore, for any $v\in H^1(\cup \Omega^\pm)$, we have
$[\![v_E^\pm]\!](\mathbf{x})=v_E^+(\mathbf{x})-v_E^-(\mathbf{x})$ for all $\mathbf{x}\in N(\Gamma,\delta_0),$
which can be viewed as an extension of the jump $[v]_\Gamma$.
Note that the difference between $[\![v_E^\pm]\!](\mathbf{x})$ and $[v]_{\Gamma}(\mathbf{x})$ is the range of $\mathbf{x}$. For vector-valued functions, the jumps $[\![\cdot]\!]$ and $[\cdot]_{\Gamma}$ are defined analogously.

Finally, we consider the extensions of the tangential gradients along the exact interface $\Gamma$ and the discrete interface $\Gamma_h$. Noting that $\mathbf{n}$ and $\mathbf{n}_h$ are well-defined in the neighborhood of $\Gamma$, for any $v\in H^1(U(\Gamma,\delta_0))$, these extensions are defined naturally as 
\begin{equation}\label{def_sur_grad}
\begin{aligned}
(\nabla_\Gamma v)(\mathbf{x})&:=\nabla v-(\mathbf{n}\cdot \nabla v)\mathbf{n}   \qquad &&\forall \mathbf{x}\in U(\Gamma,\delta_0),\\
(\nabla_{\Gamma_h} v)(\mathbf{x})&:=\nabla v-(\mathbf{n}_h \cdot \nabla v)\mathbf{n}_h\qquad &&\forall \mathbf{x}\in T,~ T\in\mathcal{T}_h^\Gamma.
\end{aligned}
\end{equation}
Let $\mathbf{t}_i(\mathbf{x})$, $i=1,...,N-1$ be  standard basis vectors in the plane perpendicular to $\mathbf{n}(\mathbf{x})$.
By definition, there hold
\begin{equation*}
\nabla_\Gamma v=\sum_{i=1}^{N-1}(\mathbf{t}_i\cdot \nabla v)\mathbf{t}_i \quad \mbox{ and }\quad\nabla_\Gamma v=\mathbf{0}  ~ \mbox{ if } v|_{\Gamma}=0.
\end{equation*}
Analogously, we have
\begin{equation}\label{sur_grad_h_ineq}
\nabla_{\Gamma_h} v=\sum_{i=1}^{N-1}(\mathbf{t}_{i,h}\cdot \nabla v)\mathbf{t}_{i,h}\quad \mbox{ and }\quad |\mathbf{t}_{h}\cdot \nabla v|\leq |\nabla_{\Gamma_h} v|,
\end{equation}
where $\mathbf{t}_{i,h}(\mathbf{x})$, $i=1,...,N-1$ and $\mathbf{n}_h(\mathbf{x})$ form standard basis vectors in $\mathbb{R}^N$ and $\mathbf{t}_{h}$ is an arbitrary unit vector perpendicular to $\mathbf{n}_h$.

\subsection{The immersed Crouzeix-Raviart finite element}
For each element $T\in\mathcal{T}_h$, we  define 
the linear functional $\mathcal{N}_{i,T}: W(T)\rightarrow \mathbb{R}$ by
\begin{equation}\label{def_Nit}
\mathcal{N}_{i,T}(v)=\frac{1}{|F_i|}\int_{F_i}v, 
\end{equation}
where $F_i$'s are faces of $T$, $|F_i|$ means the measure of $F_i$, and  
\begin{equation}\label{def_WT}
W(T)=\{v\in L^2(T) : v|_{F_i}\in L^2(F_i), ~ i=1,...,N+1 \}.
\end{equation}
The standard Crouzeix-Raviart finite element then is $( T,  \mathbb{P}_1(T),  \Sigma_T )$, where
\begin{equation}\label{standard_CR}
\Sigma_T=\{\mathcal{N}_{1,T}, \mathcal{N}_{2,T}, ..., \mathcal{N}_{N+1,T}\}.
\end{equation}

On an interface element $T\in\mathcal{T}_h^\Gamma$, in order to encode the interface jump conditions (\ref{p1.2})-(\ref{p1.3}) into finite element spaces, we replace  the shape function space $\mathbb{P}_1(T)$ by
\begin{equation}\label{def_Sh}
S_h(T):=\{\phi\in L^2(T) : \phi|_{T_h^\pm}\in  \mathbb{P}_1(T_h^\pm),~ [\phi]_{\Gamma_{h,T}}=0,~ [\beta_T\nabla \phi\cdot \mathbf{n}_h]_{\Gamma_{h,T}}=0 \},
\end{equation}
where $[\cdot]_{\Gamma_{h,T}}$ denotes the jump across $\Gamma_{h,T}$, and the function $\beta_T(\mathbf{x})$ is a piecewise constant on $T$ defined by $\beta_T|_{T_h^\pm}=\beta^\pm_T$ with the constants $\beta_T^+$ and  $\beta_T^-$ chosen such that 
\begin{equation}\label{betaT}
\|\beta_E^\pm(\mathbf{x})-\beta_T^\pm\|_{L^\infty(T)}\leq Ch_T.
\end{equation}
Obviously, $S_h(T)$ is a linear space, and we have $\mbox{dim}(S_h(T))=N+1=\mbox{card}(\Sigma_T)$. Now the immersed Crouzeix-Raviart finite element  is defined as $( T,  S_h(T),  \Sigma_T )$.

\begin{remark}
We can choose $\beta_T^\pm=\beta^\pm(\mathbf{x}_c)$ with an arbitrary point $\mathbf{x}_c\in \overline{T}$ to satisfy the requirement (\ref{betaT}) since $\beta_E^\pm(\mathbf{x})\in C^1(\overline {T})$ for all $T\in\mathcal{T}_h^\Gamma$. We emphasize that this approximation of the coefficient $\beta(\mathbf{x})$  is only used in the construction of the IFE space, not in the bilinear form of the  IFE method. To avoid integrating on curved regions, we will approximate the coefficient $\beta(\mathbf{x})$ by another function, i.e., $\beta^{BK}(\mathbf{x})$ (see Section~\ref{sec_method}).  
\end{remark}

\begin{remark}\label{remark_th}
Let $\mathbf{x}_T^P$  be an arbitrary point on the plane $\Gamma_{h,T}^{ext}$, and $\mathbf{t}_{i,h}$, $i=1,...,N-1$ be standard basis vectors in the plane perpendicular to $\mathbf{n}_h$. Then the interface condition $[\phi]_{\Gamma_{h,T}}=0$ in (\ref{def_Sh}) is  equivalent to 
\begin{equation*}
[\![\phi^\pm]\!](\mathbf{x}_T^P)=0 ~\mbox{ and }~ [\![\nabla \phi^\pm\cdot\mathbf{t}_{i,h}]\!]=0,~i=1,...,N-1.
\end{equation*}
\end{remark}

\section{Properties of the immersed finite element}\label{sec_pro}
To show that $( T,  S_h(T),  \Sigma_T )$ is indeed a finite element, we need to prove that $\Sigma_T$ determines $S_h(T)$, i.e.,  $\phi\in S_h(T)$ with $\mathcal{N}_{i,T}(\phi)=0$  $\forall \mathcal{N}_{i,T}\in \Sigma_T$ implies that $\phi=0$; see Chapter 3 in \cite{brenner2008mathematical}. Equivalently, in the next subsection we prove the existence and uniqueness of  the IFE basis functions defined by
\begin{equation}\label{def_phi}
\phi_{i,T}(\mathbf{x})\in S_h(T),~\mathcal{N}_{j,T}(\phi_{i,T})=\delta_{ij}  \mbox{ (the Kronecker symbol)}\quad \forall i,j=1,..., N+1.
\end{equation}

\subsection{Unisolvence of the basis functions}\label{subsec_uniso}
Clearly, the IFE shape function space $S_h(T)$ is not empty since $0\in S_h(T)$. Given a function $\phi\in S_h(T)$,
if we know the jump
\begin{equation}\label{def_mu}
\alpha:=[\nabla \phi\cdot\mathbf{n}_h]_{\Gamma_{h,T}},
\end{equation}
which is a constant, 
then the function $\phi$ can be written as 
\begin{equation}\label{phi_decomp}
\phi(\mathbf{x})=\phi_0(\mathbf{x})+\alpha\phi_J(\mathbf{x}),
\end{equation}
where $\phi_0(\mathbf{x})$ and $\phi_J(\mathbf{x})$ are defined by 
\begin{align}
&\phi_0|_{T_h^\pm}\in \mathbb{P}_1(T_h^\pm), ~[\phi_0]_{\Gamma_{h,T}}=0,~ [\nabla \phi_0\cdot\mathbf{n}_h]_{\Gamma_{h,T}}=0,~ \mathcal{N}_{i,T}(\phi_0)=\mathcal{N}_{i,T}(\phi),~ i=1,...,N+1,\label{def_phi0}\\
&\phi_J|_{T_h^\pm}\in \mathbb{P}_1(T_h^\pm), ~[\phi_J]_{\Gamma_{h,T}}=0, ~[\nabla \phi_J\cdot\mathbf{n}_h]_{\Gamma_{h,T}}=1, ~\mathcal{N}_{i,T}(\phi_J)=0, ~ i=1,...,N+1.\label{def_phiJ}
\end{align}
It is easy to check that 
\begin{equation}\label{phi0_pro}
\phi_0(\mathbf{x})\in\mathbb{P}_1(T)~~\mbox{and}~~\phi_0(\mathbf{x})=\sum_{i=1}^{N+1}\mathcal{N}_{i,T}(\phi)\lambda_{i,T}(\mathbf{x}),
\end{equation}
where $\lambda_{i,T}(\mathbf{x})$ is the standard Crouzeix-Raviart  basis function defined by
\begin{equation}\label{def_lambda}
\lambda_{i,T}(\mathbf{x})\in\mathbb{P}_1(T),~\mathcal{N}_{j,T}(\lambda_{i,T})=\delta_{ij},~ j=1,..., N+1.
\end{equation}
Next, we show $\phi_J(\mathbf{x})$ also exists uniquely and can be constructed explicitly. Suppose there is another function satisfying (\ref{def_phiJ}), denoted by $\tilde{\phi}_J$, then it is easy to see from (\ref{def_phiJ}) that $\phi_J-\tilde{\phi}_J=0$, which implies the uniqueness. 
Let $\Pi_T: W(T)\rightarrow \mathbb{P}_1(T)$ be the standard Crouzeix-Raviart interpolation operator defined by
\begin{equation}\label{def_inter}
\mathcal{N}_{i,T}(\Pi_T v)=\mathcal{N}_{i,T}(v),~~i=1,...,N+1.
\end{equation}
The existence can be proved by constructing the function explicitly as
\begin{equation}\label{phiJ_pro}
\phi_J(\mathbf{x})=w(\mathbf{x})-\Pi_Tw(\mathbf{x})\quad\mbox{ with }\quad w|_{T_h^+}=d_{\Gamma^{ext}_{h,T}}~\mbox{ and }~w|_{T_h^-}=0,
\end{equation}
where $d_{\Gamma^{ext}_{h,T}}$ is the signed distance function  to the plane $\Gamma^{ext}_{h,T}$, i.e.,
\begin{equation*}
d_{\Gamma^{ext}_{h,T}}(\mathbf{x})=
\left\{
\begin{aligned}
&\mbox{dist}(\mathbf{x},\Gamma^{ext}_{h,T})\qquad&&\mbox{ if }\mathbf{x}\in \overline{T_h^+},\\
&-\mbox{dist}(\mathbf{x},\Gamma^{ext}_{h,T})&&\mbox{ if }\mathbf{x}\in T_h^-.
\end{aligned}
\right.
\end{equation*}
It is easy to verify that the constructed function above indeed satisfies (\ref{def_phiJ}).

Now the problem is to find the constant $\alpha$ defined in (\ref{def_mu}). Substituting (\ref{phi_decomp}) into  the jump condition $[\beta_T\nabla \phi\cdot \mathbf{n}_h]_{\Gamma_{h,T}}=0$ in (\ref{def_Sh}), we have 
\begin{equation*}
[\beta_T\nabla \phi_J\cdot \mathbf{n}_h]_{\Gamma_{h,T}}\alpha=-[\beta_T\nabla \phi_0\cdot \mathbf{n}_h]_{\Gamma_{h,T}}.
\end{equation*}
Using (\ref{phi0_pro}) and (\ref{phiJ_pro}), we arrive at 
\begin{equation}\label{eq_mu}
\left(1+(\beta^-_T/\beta^+_T-1)\nabla \Pi_T w\cdot \mathbf{n}_h\right)\alpha=(\beta^-_T/\beta^+_T-1)\nabla \phi_0\cdot \mathbf{n}_h.
\end{equation}

To show the existence and uniqueness of the constant $\alpha$, we prove the following novel result which is the key of this paper.
\begin{lemma}\label{lem_01}
Let $T\in\mathcal{T}_h^\Gamma$ be an arbitrary triangle ($N=2$) or tetrahedron ($N=3$),  and $w$ be a piecewise linear function defined in (\ref{phiJ_pro}). Then it holds 
\begin{equation}\label{est01}
 \nabla\Pi_T w\cdot \mathbf{n}_h=\frac{|T_h^+|}{|T|}\in [0,1],
\end{equation}
where $|\cdot|$ stands for the measure of domains (i.e., area for $N=2$ and volume for $N=3$).
\end{lemma}
\begin{proof}
We give a unified proof for both $N=2$ and $N=3$ (including Type I and Type II interface elements) by using the Gauss theorem.  More precisely, by definition, we have
\begin{equation*}
\nabla\cdot(w\mathbf{n}_h)|_{T_h^+}=\nabla\cdot\left(\mbox{dist}(\mathbf{x},\Gamma_{h,T}^{ext})\mathbf{n}_h\right)=1.
\end{equation*}
Then the Gauss theorem gives  
\begin{equation*}
\int_{\partial T_h^+}w\mathbf{n}_h\cdot\boldsymbol{\nu}=\int_{T_h^+}\nabla\cdot (w\mathbf{n}_h)=|T_h^+|,
\end{equation*}
where $\boldsymbol{\nu}$ is the unit exterior normal vector to  $\partial T_h^+$.
Observing that $w=0$ on $\Gamma_{h,T}$ and $\partial T_h^+$ is composed of $\Gamma_{h,T}$ and $F_i\cap \partial T_h^+$, $i=1,...,N+1$, we obtain 
\begin{equation*}
\begin{aligned}
\sum_{i=1}^{N+1}\int_{F_i\cap \partial T_h^+ }w\mathbf{n}_h\cdot\boldsymbol{\nu}_i=\int_{\partial T_h^+}w\mathbf{n}_h\cdot\boldsymbol{\nu}=|T_h^+|,
\end{aligned}
\end{equation*}
where  $F_i$, $i=1,...,N+1$ are the faces of $T$ and $\boldsymbol{\nu}_i$ is the unit exterior normal vector to $F_i$.

On the other hand, let $l_i$ be the distance from the face $F_i$ to the opposite vertex of $T$, then by a simple calculation we have the following identity for the standard Crouzeix-Raviart basis function 
\begin{equation*}
\nabla \lambda_{i,T}=\frac{N}{l_i}\boldsymbol{\nu}_i.
\end{equation*}
Using the above two identities we can derive 
\begin{equation*}
\begin{aligned}
\nabla\Pi_Tw\cdot\mathbf{n}_h&
=\sum_{i=1}^{N+1}\mathcal{N}_{i,T}(w)\nabla \lambda_{i,T}\cdot\mathbf{n}_h\\
&=\sum_{i=1}^{N+1}\frac{1}{|F_i|}\left(\int_{F_i\cap \partial T_h^+ }w\right) \frac{N}{l_i}\boldsymbol{\nu}_i\cdot\mathbf{n}_h\\
&=\sum_{i=1}^{N+1}\frac{1}{|T|}\int_{F_i\cap \partial T_h^+ }w\boldsymbol{\nu}_i\cdot\mathbf{n}_h=\frac{|T_h^+|}{|T|},
\end{aligned}
\end{equation*}
which completes the proof of this lemma.
\end{proof}

\begin{theorem}\label{theo_basis}
For any $T\in\mathcal{T}_h^\Gamma$, the IFE basis functions defined in (\ref{def_phi}) exist uniquely and have the following explicit formula  
\begin{equation}\label{express_IFE_basis}
\phi_{i,T}(\mathbf{x})=\lambda_{i,T}(\mathbf{x})+\frac{(\beta^-_T/\beta^+_T-1)\nabla \lambda_{i,T} \cdot \mathbf{n}_h}{1+(\beta^-_T/\beta^+_T-1)|T_h^+|/|T|}(w(\mathbf{x})-\Pi_Tw(\mathbf{x})),\qquad i=1,...,N+1,
\end{equation}
where $\lambda_{i,T}$ is the standard Crouzeix-Raviart  basis function defined in (\ref{def_lambda}), the function $w$ is a piecewise linear function defined in (\ref{phiJ_pro}), and $\Pi_T$ is the standard Crouzeix-Raviart interpolation operator defined in (\ref{def_inter}).
\end{theorem}

\begin{proof}
Using Lemma~\ref{lem_01} we have
\begin{equation}\label{pro_basi_fenmu}
1+(\beta^-_T/\beta^+_T-1)\nabla\Pi_Tw\cdot \mathbf{n}_h\geq\left\{
\begin{aligned}
&1\qquad&&\mbox{ if }\beta^-_T/\beta^+_T\geq 1,\\
&\beta^-_T/\beta^+_T \qquad&&\mbox{ if }0<\beta^-_T/\beta^+_T< 1,
\end{aligned}\right.
\end{equation}
which implies that the equation (\ref{eq_mu})  has a unique solution
\begin{equation*}
\alpha= \frac{(\beta^-_T/\beta^+_T-1)\nabla\phi_0 \cdot \mathbf{n}_h}{1+(\beta^-_T/\beta^+_T-1)|T_h^+|/|T|}.
\end{equation*}
Substituting this identity, (\ref{phi0_pro}) and (\ref{phiJ_pro})  into (\ref{phi_decomp}) yields 
\begin{equation}\label{pro_rela_ife_cr}
\phi(\mathbf{x})=\sum_{j=1}^{N+1}\mathcal{N}_{j,T}(\phi)\lambda_{j,T}(\mathbf{x})+\frac{(\beta^-_T/\beta^+_T-1)\sum_{j=1}^{N+1}\mathcal{N}_{j,T}(\phi)\nabla \lambda_{j,T} \cdot \mathbf{n}_h}{1+(\beta^-_T/\beta^+_T-1)|T_h^+|/|T|}(w(\mathbf{x})-\Pi_Tw(\mathbf{x})).
\end{equation}
The desired result (\ref{express_IFE_basis}) follows from the above identity and the definition (\ref{def_phi}).
\end{proof}

\begin{remark}
We  highlight  that  the denominator in  (\ref{express_IFE_basis}) does not approach zero  even if $|T_h^+|\rightarrow 0$ or $|T_h^-|\rightarrow 0$. We find $(w-\Pi_T w)|_T\rightarrow0$ as $|T^+_h|\rightarrow0$ or $|T^-_h|\rightarrow0$. Thus, from (\ref{express_IFE_basis})  we claim that the IFE basis functions tend to  the standard finite element  basis functions, i.e.,  $\phi_{i,T}\rightarrow \lambda_{i,T}$ as $|T^+_h|\rightarrow0$ or $|T^-_h|\rightarrow0$. On the other hand, it is easy to see that $\phi_{i,T}\rightarrow \lambda_{i,T}$ as $\beta_T^+-\beta_T^-\rightarrow0$.  This consistency of the IFE  with the standard finite element  is different from other unfitted mesh methods (see, e.g.,\cite{hansbo2002unfitted,2016High,burman2015cutfem}). This nice property of the IFE  is desirable for moving interface problems \cite{guo2021SIAM} and interface inverse problems \cite{guo2019fixed}.
\end{remark}

\subsection{Bounds for the basis functions}
It is obvious that the standard Crouzeix-Raviart  basis functions satisfy the following estimates 
\begin{equation}\label{esti_lambda}
|\lambda_{i,T}|_{W_\infty^m(T)}\leq Ch_T^{-m},\qquad i=1,...,N+1,~m=0,1,
\end{equation}
where the constant $C$ only depends on the shape regularity parameter $\varrho$.  In this subsection we show that the IFE basis functions also have similar bounds with a constant independent of the interface location relative to the mesh. This property  plays an important role in the theoretical analysis.
\begin{theorem}\label{lem_est_IFEbasis}
There exists a constant $C$, depending only on $\beta_T^\pm$ and the shape regularity parameter $\varrho$, such that for all $T\in\mathcal{T}_h^\Gamma$,
\begin{equation}\label{est_IFEbasis}
|\phi_{i,T}^\pm|_{W^m_\infty(T)}\leq Ch_T^{-m}, \quad i=1,..., N+1, ~m=0,1.
\end{equation}
\end{theorem}
\begin{proof}
In view of (\ref{express_IFE_basis}) and (\ref{phiJ_pro})  we have 
\begin{equation}\label{pro_esti_IFEbasis1}
\begin{aligned}
&\phi_{i,T}^+(\mathbf{x})=\lambda_{i,T}(\mathbf{x})+\frac{(\beta^-_T/\beta^+_T-1)\nabla \lambda_{i,T} \cdot \mathbf{n}_h}{1+(\beta^-_T/\beta^+_T-1)|T_h^+|/|T|}(d_{\Gamma^{ext}_{h,T}}(\mathbf{x})-\Pi_Tw(\mathbf{x}))\quad &&\forall \mathbf{x}\in T,\\
&\phi_{i,T}^-(\mathbf{x})=\lambda_{i,T}(\mathbf{x})-\frac{(\beta^-_T/\beta^+_T-1)\nabla \lambda_{i,T} \cdot \mathbf{n}_h}{1+(\beta^-_T/\beta^+_T-1)|T_h^+|/|T|}\Pi_Tw(\mathbf{x})\quad &&\forall \mathbf{x}\in T.\\
\end{aligned}
\end{equation}
By definition, it is easy to verify 
\begin{equation}\label{pro_esti_IFEbasis2}
\|d_{\Gamma^{ext}_{h,T}}\|_{L^\infty(T)}\leq Ch_T, ~~~|d_{\Gamma^{ext}_{h,T}}|_{W_\infty^1(T)}=1
\end{equation}
and
\begin{equation}\label{pro_esti_IFEbasis3}
\begin{aligned}
|\Pi_T w|_{W_\infty^m(T)}&=\left|\sum_{i=1}^{N+1} \lambda_{i,T}\frac{1}{|F_i|}\int_{F_i} w\right|_{W_\infty^m(T)}\\
&\leq Ch_T\sum_{i=1}^{N+1}|\lambda_{i,T}|_{W_\infty^m(T)}\\
&\leq Ch^{1-m}_T.
\end{aligned}
\end{equation}
Combining (\ref{pro_esti_IFEbasis1})-(\ref{pro_esti_IFEbasis3}), (\ref{pro_basi_fenmu}) and (\ref{esti_lambda}) yields the desired result (\ref{est_IFEbasis}).
\end{proof}

\subsection{Bounds for the interpolation errors}
In this subsection we prove the optimal  IFE interpolation error estimates. To begin with, we summary the following useful properties of the standard Crouzeix-Raviart interpolation operator $\Pi_T$:
\begin{align}
&|\Pi_T v-v|_{H^m(T)}\leq Ch_T^{2-m}|v|_{H^2(T)},~m=0,1,\label{pro_PI_1}\\
&\|\Pi_T v-v\|_{L^\infty(T)}\leq Ch_T^{2-N/2}|v|_{H^2(T)},\label{pro_PI_2}\\
&|\Pi_T v|_{H^1(T)}\leq C|v|_{H^1(T)},\label{pro_PI_3}
\end{align}
which are fundamental results in the finite element analysis. Here we emphasize that the interpolation operator $\Pi_T$ is defined based on the integral values on edges so that the estimate (\ref{pro_PI_3}) holds.  It is worth noting that we only use the above properties of the operator $\Pi_T$, so we can replace it by another operator satisfying the same properties, for example, the $L^2$ projection onto $\mathbb{P}_1(T)$.

In the analysis, we need a broken operator $E^{BK}_h: H^m(\cup \Omega^\pm)\rightarrow H^m(\cup \Omega_h^\pm)$ for any $m \geq 0$ defined by
\begin{equation}\label{def_BK}
(E^{BK}_hv)|_{\Omega_h^\pm}=v_E^\pm.
\end{equation}
Similarly, the operator $\Pi_T^{BK} : H^k(\cup T^\pm)\rightarrow H^k(\cup T_h^\pm)$ is defined  by
\begin{equation}\label{def_BK2}
(\Pi_T^{BK}v)|_{T_h^\pm}=\Pi_Tv_E^\pm.
\end{equation}
Let $\widetilde{\Pi}_T^{\rm IFE}: W(T)\rightarrow S_h(T)$ be the local IFE interpolation operator  defined by
\begin{equation*}
\mathcal{N}_{i,T}(\widetilde{\Pi}_T^{\rm IFE} v)=\mathcal{N}_{i,T}(v),~~i=1,...,N+1.
\end{equation*}
We also need the IFE interpolation operator $\Pi_T^{\rm IFE}:=\widetilde{\Pi}_T^{\rm IFE}E^{BK}_h$.
Obviously,
\begin{equation}\label{def_inter_perp}
\mathcal{N}_{i,T}(\Pi_T^{\rm IFE} v)=\mathcal{N}_{i,T}(\widetilde{\Pi}_T^{\rm IFE} E^{BK}_hv)=\mathcal{N}_{i,T}(E^{BK}_hv).
\end{equation}
 
For each $F\in\mathcal{F}_h^\circ$, denote by $\mathbf{n}_F$ a unit vector normal to $F$ and  let $T^F_1$ and $T^F_2$ be two elements sharing the common face $F$ such that $\mathbf{n}_F$ points from $T^F_1$ to $T^F_2$.
The jump across the face is denoted by  $[v]_F:=v|_{T_1^F}-v|_{T_2^F}.$
When $F\in \mathcal{F}_h^b$,~ $\mathbf{n}_F$ is the unit outward normal vector of $\partial \Omega$ and $[v]_F:=v$.
We then define the global IFE space by
\begin{equation*}
\begin{split}
V^{\rm IFE}_h:=\{v: v|_T\in \mathbb{P}_1(T)~\forall T\in\mathcal{T}_h^{non},~ v|_T\in S_h(T)~\forall T\in\mathcal{T}_h^{\Gamma}, \int_{F}[v]_F=0~\forall F\in \mathcal{F}^\circ  \}
\end{split}
\end{equation*}
and  the global IFE interpolation operator $\Pi_h^{\rm IFE}: \widetilde{H}^2(\Omega)\rightarrow V^{\rm IFE}_h$ by
\begin{equation*}
(\Pi_h^{\rm IFE}v)|_{T}=\left\{
\begin{aligned}
&\Pi_T^{\rm IFE}v \quad &&\mbox{ if } T\in\mathcal{T}_h^\Gamma, \\
&\Pi_T v  \quad &&\mbox{ if } T\in\mathcal{T}_h^{non}.
\end{aligned}
\right.
\end{equation*}
Analogously, the standard Crouzeix-Raviart finite element space $V_h$ and the interpolation operator $\Pi_h: \widetilde{H}^2(\Omega)\rightarrow V_h$ are defined by 
$V_h:=\{v: v|_T\in \mathbb{P}_1(T)~\forall T\in\mathcal{T}_h, \int_{F}[v]_F=0~\forall F\in \mathcal{F}^\circ  \}$ and  $(\Pi_hv)|_{T}=\Pi_T(v|_{T})$, respectively.

For clarity, we outline our approach for deriving the bounds of the interpolation errors. We aim to estimate
\begin{equation}\label{outline_1}
\begin{aligned}
\left|E^{BK}_hv- \Pi_{T}^{\rm IFE}v\right|^2_{H^m(\cup T_h^\pm)}\leq \sum_{s=\pm}|v_E^s- \left(\Pi_{T}^{\rm IFE}v\right)^s|^2_{H^m(T)}\quad \forall T\in\mathcal{T}_h^\Gamma.
\end{aligned}
\end{equation}
Obviously, we have the split
\begin{equation}\label{outline_2}
\begin{aligned}
|v_E^\pm- \left(\Pi_{T}^{\rm IFE}v\right)^\pm|_{H^m(T)}
\leq \underbrace{|v_E^\pm-\Pi_{T}v_E^\pm|_{H^m(T)}}_{({\rm I})_1}+\underbrace{|\Pi_{T}v_E^\pm- \left(\Pi_{T}^{\rm IFE}v\right)^\pm|_{H^m(T)}}_{({\rm I})_2}.
\end{aligned}
\end{equation}
The estimate of the first term $({\rm I})_1$ follows directly from  (\ref{pro_PI_1}) and the main difficulty is to estimate the second term $({\rm I})_2$. 
Noticing that the function in $({\rm I})_2$ is piecewise linear on $T_h^\pm$, our idea is to decompose it by proper degrees of freedom (see Lemma~\ref{lem_decomp}), and then estimate each terms in the decomposition (see Theorem~\ref{theo_T_inter}). The degrees of freedom for determining the function in the term $({\rm I})_2$ should include $\mathcal{N}_{i,T}$, $i=1,...,N+1$, and others related to the interface jumps, which inspire us to define the novel auxiliary functions $\Psi_T$, $\Upsilon_{T}$ and $\Theta_{i,T}$, $i=1,...,N-1$ as follows.

On each interface element $T$,  the  auxiliary functions $\Psi_T$, $\Upsilon_{T}$ and $\Theta_{i,T}$, $i=1,...,N-1$, are defined such that 
\begin{equation}\label{def_auxi1}
\begin{aligned}
&\Psi_T|_{T_h^\pm}\in\mathbb{P}_1(T_h^\pm) ,~\Upsilon_{T}|_{T_h^\pm}\in\mathbb{P}_1(T_h^\pm),~\Theta_{i,T}|_{T_h^\pm}\in\mathbb{P}_1(T_h^\pm),\\
&\mathcal{N}_{j,T}(\Psi_T)=\mathcal{N}_{j,T}(\Upsilon_{T})=\mathcal{N}_{j,T}(\Theta_{i,T})=0,~j=1,...,N+1,
\end{aligned}
\end{equation}
and
\begin{equation}\label{def_auxi2}
\begin{aligned}
&[\![\Psi_T^\pm]\!](\mathbf{x}_T^P)=1,~&&[\![\beta_T^\pm \nabla\Psi_T^\pm \cdot \textbf{n}_h]\!]=0,~ &&[\![\nabla\Psi_T^\pm \cdot \textbf{t}_{j,h}]\!]=0,&& j=1,...,N-1,\\
&[\![\Upsilon_{T}^\pm]\!](\mathbf{x}_T^P)=0,~&&[\![\beta_T^\pm \nabla\Upsilon_{T}^\pm \cdot \textbf{n}_h]\!]=1,~ &&[\![\nabla\Upsilon_{T}^\pm \cdot \textbf{t}_{j,h}]\!]=0,&& j=1,...,N-1,\\
&[\![\Theta_{i,T}^\pm]\!](\mathbf{x}_T^P)=0,~&&[\![\beta_T^\pm \nabla\Theta_{i,T}^\pm \cdot \textbf{n}_h]\!]=0,~ &&[\![\nabla\Theta_{i,T}^\pm \cdot \textbf{t}_{j,h}]\!]=\delta_{ij},&& j=1,...,N-1,
\end{aligned}
\end{equation}
where $\mathbf{t}_{j,h}$, $j=1,...,N-1$ are defined in Remark~\ref{remark_th} and the point $\mathbf{x}_T^P$ belonging to the plane $\Gamma_{h,T}^{ext}$ is chosen carefully as follows. 
We set $\mathbf{x}_T^P=\mathbf{p}_{\Gamma_{h,T}^{ext}}(\mathbf{x}_T^S)$, where $\mathbf{x}_T^S$ is an arbitrary point on the surface $\Gamma_{T}$ and  $\mathbf{p}_{\Gamma_{h,T}^{ext}}$ is the  orthogonal projection onto the plane $\Gamma_{h,T}^{ext}$.
From  (\ref{ass_Gamma_h_2}) we have the relation 
\begin{equation}\label{tpts}
\left|\mathbf{x}_T^P-\mathbf{x}_T^S\right|\leq \|\mbox{dist}(\mathbf{x},\Gamma_{h,T}^{ext})\|_{L^\infty(\Gamma_T)}\leq Ch_T^2.
\end{equation}

\begin{remark}\label{remark_xtxp}
We note that the point $\mathbf{x}_T^P\in \Gamma^{ext}_{h,T}$ may not belong to the planar segment $\Gamma_{h,T}$ (see Figure~\ref{fig_inte_2d} for the 2D case). This choice of $\mathbf{x}_T^P$ and $\mathbf{x}_T^S$ is crucial in deriving the bound for $|a_T|$ in (\ref{xts_dif_0}). Otherwise, if we choose  $\mathbf{x}_T^S=\mathbf{p}(\mathbf{x}_T^P)$ with a point $\mathbf{x}_T^P\in \Gamma_{h,T}$, although the relation (\ref{tpts}) also holds, the point $\mathbf{x}_T^S$ may be  outside of $T$, which brings difficulties in the analysis (see (\ref{xts_dif})).
\end{remark}

\begin{lemma}\label{lem_arxi}
The functions $\Psi_T$, $\Upsilon_{T}$ and $\Theta_{i,T}$, $i=1,...N-1$, exist uniquely and satisfy  
\begin{equation}\label{est_psi}
|\Psi^\pm_{T}|_{W^m_\infty(T)}\leq Ch_T^{-m},\quad |\Upsilon^\pm_{T}|_{W^m_\infty(T)}\leq Ch_T^{1-m},\quad |\Theta^\pm_{i,T}|_{W^m_\infty(T)}\leq Ch_T^{1-m},~~m=0,1,
\end{equation}
where the constant $C$ depends only on $\beta_T^\pm$ and the shape regularity parameter $\varrho$.
\end{lemma}
\begin{proof}
First we  prove the uniqueness.  Suppose there is another function, denoted by $\tilde{\Psi}_T$, satisfying the same conditions as $\Psi_T$ in (\ref{def_auxi1})-(\ref{def_auxi2}). Then it is easy to see
\begin{equation*}
[\![(\tilde{\Psi}_T-\Psi_T)^\pm]\!](\mathbf{x}_T^P)=0,~[\![\beta_T^\pm\nabla(\tilde{\Psi}_T-\Psi_T)^\pm\cdot\mathbf{n}_h]\!]=0,~[\![\nabla(\tilde{\Psi}_T-\Psi_T)^\pm\cdot\mathbf{t}_{i,h}]\!]=0,
\end{equation*}
which leads to $\tilde{\Psi}_T-\Psi_T\in S_h(T)$. By definition, we also have $\mathcal{N}_{i,T}(\tilde{\Psi}_T-\Psi_T)=0$, $i=1,...,N+1$. Thus, by the unisolvence of the basis functions proved in Section~\ref{subsec_uniso}, we obtain  $\tilde{\Psi}_T-\Psi_T=0$, which implies that $\Psi_T$ is unique. Same analysis is valid to prove the uniqueness of $\Upsilon_T$ and $\Theta_{i,T}$. 

Next, we derive the estimates in (\ref{est_psi}). Obviously, $\Psi_T(\mathbf{x})$ can be constructed explicitly as 
\begin{equation*}
\Psi_T(\mathbf{x})=z(\mathbf{x})-\widetilde{\Pi}^{\rm IFE}_Tz(\mathbf{x})\quad\mbox{ with }\quad
z(\mathbf{x})=\left\{
\begin{aligned}
&1\quad &&\mbox{ if } \mathbf{x}\in T_h^+,\\
&0&&\mbox{ if } \mathbf{x}\in T_h^-.
\end{aligned} 
\right.
\end{equation*}
We have
\begin{equation*}
\begin{aligned}
\|z^\pm\|_{L^\infty(T)}\leq 1,~~|z^\pm|_{W^1_\infty(T)}=0,~~\left|\mathcal{N}_{i,T}(z)\right|\leq 1,~i=1,...,N+1,
\end{aligned}
\end{equation*}
which together with  (\ref{est_IFEbasis}) leads to the first inequality in (\ref{est_psi}), i.e.,
\begin{equation*}
\begin{aligned}
|\Psi_T^\pm|_{W^m_\infty(T)}&\leq |z^\pm|_{W^m_\infty(T)}+ \sum_{i=1}^{N+1}\left|\mathcal{N}_{i,T}(z)\right||\phi_{i,T}^\pm  |_{W^m_\infty(T)}\leq Ch^{-m}.
\end{aligned}
\end{equation*}
To derive the estimate for $\left|\Upsilon_T^\pm\right|_{W^m_\infty(T)}$, we construct $\Upsilon_T(\mathbf{x})$ as
\begin{equation*}
\Upsilon_{T}(\mathbf{x})=z(\mathbf{x})-\widetilde{\Pi}^{\rm IFE}_Tz(\mathbf{x})\quad\mbox{ with }\quad
z(\mathbf{x})=\left\{
\begin{aligned}
&\frac{1}{\beta_T^+}(\mathbf{x}-\mathbf{x}_T^P)\cdot\mathbf{n}_h\quad &&\mbox{ if } \mathbf{x}\in T_h^+,\\
&0&&\mbox{ if } \mathbf{x}\in T_h^-.
\end{aligned} 
\right.
\end{equation*}
To deal with that the point $\mathbf{x}_T^P$ may not belong to $T$ (see Remark~\ref{remark_xtxp}),  we use the inequality (\ref{tpts}) and  fact $\mathbf{x}_T^S\in\Gamma_{T}\subset T$ to get
\begin{equation*}
|\mathbf{x}-\mathbf{x}_T^P|\leq |\mathbf{x}-\mathbf{x}_T^S|+|\mathbf{x}_T^S-\mathbf{x}_T^P|\leq Ch_T\qquad \forall \mathbf{x}\in T.
\end{equation*}
Now we have
\begin{equation*}
\begin{aligned}
\|z^\pm\|_{L^\infty(T)}\leq Ch_T,~~|z^\pm|_{W^1_\infty(T)}\leq C,~~\left|\mathcal{N}_{i,T}(z)\right|\leq Ch_T,~i=1,...,N+1,
\end{aligned}
\end{equation*}
which together with (\ref{est_IFEbasis}) leads to the desired result 
 \begin{equation*}
 \begin{aligned}
|\Upsilon_T^\pm|_{W^m_\infty(T)}&\leq |z^\pm|_{W^m_\infty(T)}+ \sum_{i=1}^{N+1}\left|\mathcal{N}_{i,T}(z)\right||\phi_{i,T}^\pm  |_{W^m_\infty(T)}\leq Ch^{1-m}.
\end{aligned}
\end{equation*}
Same analysis is valid to prove the third inequality in (\ref{est_psi})   if we construct $\Theta_{i,T}$ as
\begin{equation*}
\Theta_{i,T}(\mathbf{x})=z(\mathbf{x})-\widetilde{\Pi}^{\rm IFE}_Tz(\mathbf{x})\quad\mbox{ with }\quad
z(\mathbf{x})=\left\{
\begin{aligned}
&(\mathbf{x}-\mathbf{x}_T^P)\cdot\mathbf{t}_{i,h}\quad &&\mbox{ if } \mathbf{x}\in T_h^+,\\
&0&&\mbox{ if } \mathbf{x}\in T_h^-.
\end{aligned} 
\right.
\end{equation*}
\end{proof}
 
\begin{lemma}\label{lem_decomp}
For all $T\in\mathcal{T}_h^\Gamma$ and  $v\in \widetilde{H}^2(T)$, we have the following decomposition 
\begin{equation*}
\Pi_T v_E^\pm-\left(\Pi_T^{\rm IFE}v\right)^\pm = a_T \Psi_{T}^\pm+b_{T} \Upsilon_{T}^\pm+\sum_{i=1}^{N-1}c_{i,T} \Theta_{i,T}^\pm+\sum_{i=1}^{N+1}g_{i,T}\phi_{i,T}^\pm,
\end{equation*}
where the constants $a_T$, $b_{T}$, $c_{i,T}$ and $g_{i,T}$ are defined as
\begin{equation}\label{coeff_abc}
\begin{aligned}
&a_T:=[\![\Pi_T v_E^\pm]\!](\mathbf{x}_T^P),~~~~&&b_{T}:=[\![\beta_T^\pm \nabla(\Pi_T v_E^\pm)\cdot \mathbf{n}_h]\!],\\
&c_{i,T}:=[\![\nabla(\Pi_T v_E^\pm)\cdot \mathbf{t}_{i,h}]\!],~~~~&&g_{i,T}:=\mathcal{N}_{i,T}(\Pi_T^{BK}v-E^{BK}_hv).
\end{aligned}
\end{equation}
\end{lemma}
\begin{proof}
Let $z|_{T_h^\pm}:= z^\pm$ with $ z^\pm:=\Pi_T v_E^\pm-\left(\Pi_T^{\rm IFE}v\right)^\pm$, then we have $ z=\Pi_T^{BK} v-\Pi_T^{\rm IFE}v$. Define another function as
\begin{equation}\label{pro_decomp1}
\begin{aligned}
\tilde{ z}:=&[\![ z^\pm]\!](\mathbf{x}_T^P)\Psi_{T}+ [\![\beta_T^\pm \nabla z^\pm\cdot \mathbf{n}_h]\!]\Upsilon_{T}+\sum_{i=1}^{N-1}[\![\nabla z^\pm\cdot \mathbf{t}_{i,h}]\!]\Theta_{i,T}+\sum_{i=1}^{N+1}\mathcal{N}_{i,T}( z)\phi_{i,T}.
\end{aligned}
\end{equation}
Next, we prove $\tilde{ z}=z$. It is easy to verify 
\begin{equation*}
[\![(\tilde{ z}- z)^\pm]\!](\mathbf{x}_T^P)=0,~
[\![\beta_T^\pm \nabla(\tilde{ z}- z)^\pm\cdot \mathbf{n}_h]\!]=0,~[\![\nabla(\tilde{ z}- z)^\pm\cdot \mathbf{t}_{i,h}]\!]=0,~ i=1,...,N-1,
\end{equation*}
which together with the definition of $S_h(T)$  (see also Remark~\ref{remark_th}) implies that $\tilde{ z}- z\in S_h(T)$. From (\ref{pro_decomp1}) we also have $\mathcal{N}_{i,T}(\tilde{ z}- z)=0$. Hence, by the unisolvence of the basis functions proved in Section~\ref{subsec_uniso}, we obtain $ z=\tilde{ z}$. 

It remains to evaluate the coefficients in (\ref{pro_decomp1}). Since $\Pi_T^{\rm IEF}v\in S_h(T)$, we have $[\![(\Pi_T^{\rm IEF}v)^\pm]\!](\mathbf{x}_T^P)=0$. Thus, we obtain
\begin{equation*}
\begin{aligned}
[\![ z^\pm]\!](\mathbf{x}_T^P)&=[\![\Pi_T v_E^\pm]\!](\mathbf{x}_T^P)-[\![(\Pi_T^{\rm IEF}v)^\pm]\!](\mathbf{x}_T^P)\\
&=[\![\Pi_T v_E^\pm]\!](\mathbf{x}_T^P)=a_T.
\end{aligned}
\end{equation*} 
The proof for $b_{T}$ and $c_{i,T}$ is similar since  $[\![\beta_T^\pm\nabla (\Pi_T^{\rm IEF}v)^\pm\cdot\mathbf{n}_{h}]\!]=[\![\nabla (\Pi_T^{\rm IEF}v)^\pm\cdot\mathbf{t}_{i,h}]\!]=0$ from the definition of $S_h(T)$  (see also Remark~\ref{remark_th}).
For the constants $g_{i,T}$, using (\ref{def_inter_perp}) we get
\begin{equation*}
\begin{aligned}
\mathcal{N}_{i,T}( z)&=\mathcal{N}_{i,T}(\Pi_T^{BK} v-\Pi_T^{\rm IFE}v)\\
&=\mathcal{N}_{i,T}(\Pi_T^{BK}v-E^{BK}_hv)=g_{i,T}.
\end{aligned}
\end{equation*}
This completes the proof of this lemma.
\end{proof}

\begin{theorem}\label{theo_T_inter}
For all $T\in\mathcal{T}_h^\Gamma$ and  $v\in \widetilde{H}^2(T)$,  there exists a constant $C$ independent of $h$ and the interface location relative to the mesh such that 
\begin{equation}\label{T_inter}
\begin{aligned}
|v_E^\pm-\left(\Pi_T^{\rm IFE}v\right)^\pm|^2_{H^m(T)}&\leq Ch_T^{4-2m}\sum_{s=\pm}\left(|v_E^s|^2_{H^1(T)}+|v_E^s|^2_{H^2(T)}\right)\\
&+Ch_T^{2-2m}\left(\left\|[\![\beta_E^\pm\nabla v_E^\pm\cdot\mathbf{n}]\!]\right\|^2_{L^2(T)}+\left\|[\![\nabla_\Gamma v_E^\pm]\!]\right\|^2_{L^2(T)}\right),~ m=0,1.
\end{aligned}
\end{equation}
\end{theorem}

Before proceeding with the proof, it is worth noting the following remark of this theorem.
\begin{remark}\label{remark_gamm_depend}
Since $v\in \widetilde{H}^2(T)$, by the definition (\ref{def_H2}) we have $[v]_{\Gamma_T}=[\beta\nabla v\cdot\mathbf{n}]_{\Gamma_T}=0$. This leads to 
\begin{equation}\label{jumpeq0}
[\![\nabla_\Gamma v_E^\pm]\!]|_{\Gamma_T}=\mathbf{0} ~~\mbox{ and }~~ [\![\beta_E^\pm\nabla v_E^\pm\cdot\mathbf{n}]\!]|_{\Gamma_T}=0.
\end{equation}
Therefore, it seems that we can obtain the optimal interpolation error estimate on each interface element by applying the following type of the Poincar\'e-Friedrichs  inequality to the second term on the right-hand side of (\ref{T_inter}),  
\begin{equation*}
\|v\|_{L^2(T)} \leq C_{P}h_T|v|_{H^1(T)} \quad \mbox{ for all } v\in H^1(T)\mbox{ with }v|_{\Gamma_T}=0.
\end{equation*}
Unfortunately, we cannot show that the constant $C_{P}$ is independent of the interface location relative to $T$. The proof of the above inequality is given as follows. Let $\bar{v}$ be the mean of $v$ on $T$, then it holds 
\begin{equation*}
\begin{aligned}
\|v\|_{L^2(T)} &\leq \|v-\bar{v}\|_{L^2(T)}+\|\bar{v}\|_{L^2(T)}\\
&\leq Ch_T|v|_{H^1(T)}+Ch_T^{N/2}|\Gamma_T|^{-1}\left|\int_{\Gamma_T}\bar{v}\right|.
\end{aligned}
\end{equation*}
Since $v|_{\Gamma_T}=0$, we can see
\begin{equation*}
\begin{aligned}
\left|\int_{\Gamma_T}\bar{v}\right|
&=\left|\int_{\Gamma_T}(\bar{v}-v)\right|\leq |\Gamma_T|^{1/2}\|\bar{v}-v\|^2_{L^2(\Gamma_T)}\\
&\leq C|\Gamma_T|^{1/2}\left(h_T^{-1/2}\|\bar{v}-v\|_{L^2(T)}+h_T^{1/2}|\bar{v}-v|_{H^1(T)}\right)\\
&\leq C|\Gamma_T|^{1/2}h_T^{1/2}|v|_{H^1(T)},
\end{aligned}
\end{equation*}
where in second inequality we have used the well-known trace inequality on the interface (see, e.g., \cite{hansbo2002unfitted,2016High}).  Combining the above inequalities yields 
\begin{equation*}
\|v\|_{L^2(T)} \leq \underbrace{ C\left(\sqrt{h_T^{N-1}|\Gamma_T|^{-1} }+1\right)}_{:=C_P}h_T|v|_{H^1(T)},
\end{equation*}
which implies that $C_P\rightarrow \infty$ as $|\Gamma_T|\rightarrow 0$. 
\end{remark}

To overcome the difficulty shown in Remark~\ref{remark_gamm_depend}, in the following theorem we take all the  interface elements together and carry out the analysis on a tubular neighborhood  of the interface $\Gamma$.
\begin{theorem}
For any $v\in  \widetilde{H}^2(\Omega)$, there exists a constant $C$ independent of $h$ and the interface location relative to the mesh such that 
\begin{align}\
&\sum_{T\in\mathcal{T}_h^\Gamma}|v_E^\pm-\left(\Pi_h^{\rm IFE}v\right)^\pm|^2_{H^m(T)}\leq Ch_\Gamma^{4-2m}\|v\|^2_{H^2(\cup\Omega^\pm)},~~m=0,1,\label{IFE_inter_gamma}\\
&\sum_{T\in\mathcal{T}_h}\left|v-\Pi_h^{\rm IFE}v\right|^2_{H^m(T)}\leq Ch^{4-2m}\|v\|^2_{H^2(\cup\Omega^\pm)},~~m=0,1.\label{IFE_inter_all}
\end{align}
\end{theorem}
\begin{proof}
Noticing that $T\subset U(\Gamma, h_\Gamma)$ for all $T\in\mathcal{T}_h^\Gamma$, combining  (\ref{jumpeq0}) and Lemma~\ref{lem_strip} we have 
\begin{equation}\label{zhengti1}
\begin{aligned}
\sum_{T\in\mathcal{T}_h^\Gamma}\left\|[\![\beta_E^\pm\nabla v_E^\pm\cdot\mathbf{n}]\!]\right\|^2_{L^2(T)}
&\leq \left\|[\![\beta_E^\pm\nabla v_E^\pm\cdot\mathbf{n}]\!]\right\|^2_{L^2(U(\Gamma, h_\Gamma))}\\
&\leq Ch_\Gamma^2\left|[\![\beta_E^\pm\nabla v_E^\pm\cdot\mathbf{n}]\!]\right|^2_{H^1(U(\Gamma, h_\Gamma))}\\
&\leq Ch_\Gamma^2\sum_{s=\pm}\left(|v_E^s|^2_{H^1(U(\Gamma, h_\Gamma))}+|v_E^s|^2_{H^2(U(\Gamma, h_\Gamma))}\right),
\end{aligned}
\end{equation}
\begin{equation}\label{zhengti2}
\begin{aligned}
\sum_{T\in\mathcal{T}_h^\Gamma}\left\|[\![\nabla_\Gamma v_E^\pm]\!]\right\|^2_{L^2(T)}
&\leq Ch_\Gamma^2 \left|[\![\nabla_\Gamma v_E^\pm]\!]\right|^2_{H^1(U(\Gamma, h_\Gamma))}\\
&\leq Ch_\Gamma^2\left|[\![\nabla v_E^\pm-(\mathbf{n}\cdot \nabla v_E^\pm)\mathbf{n} ]\!]\right|^2_{H^1(U(\Gamma, h_\Gamma))}\\
&\leq Ch_\Gamma^2\sum_{s=\pm}\left(|v_E^s|^2_{H^1(U(\Gamma, h_\Gamma))}+|v_E^s|^2_{H^2(U(\Gamma, h_\Gamma))}\right),
\end{aligned}
\end{equation}
where we have used  (\ref{beta_ext1})-(\ref{beta_ext2})  and  $\mathbf{n}(\mathbf{x})\in C^1\left( U(\Gamma,\delta_0) \right)^N$ in the derivation. Therefore, the result (\ref{IFE_inter_gamma}) follows from (\ref{zhengti1}), (\ref{zhengti2}), Theorem~\ref{theo_T_inter} and the extension result (\ref{extension}). 

To prove the estimate (\ref{IFE_inter_all}), we need to consider the mismatch region caused by approximating $\Gamma$ by $\Gamma_h$. The triangle inequality gives
\begin{equation}\label{pro_int_newcoo_1}
\begin{aligned}
\sum_{T\in\mathcal{T}_h} |v-\Pi_h^{\rm IFE}v|^2_{H^m(T)}\leq C\sum_{T\in\mathcal{T}_h}\left|E^{BK}_hv- \Pi_{h}^{\rm IFE}v\right|^2_{H^m(\cup T_h^\pm)}+C\left|v-E^{BK}_hv\right|^2_{H^m(\cup \Omega_h^\pm)}.
\end{aligned}
\end{equation}
Recalling the relation  (\ref{outline_1}), the first term can be estimated by   (\ref{IFE_inter_gamma}) for interface elements and the standard interpolation error estimates for non-interface elements. Thus, it suffices to estimate the second term on the right-hand side of (\ref{pro_int_newcoo_1}). By the definition of $E_h^{BK}$ in (\ref{def_BK}) and the relation  (\ref{tdelta_U}) we have
\begin{equation*}
\left|v-E^{BK}_hv\right|^2_{H^m(\cup\Omega_h^\pm)}=\left| [\![v_E^\pm]\!] \right|^2_{H^m(\Omega^\triangle)}\leq \left|[\![v_E^\pm]\!]\right|^2_{H^m(U(\Gamma,C_\Gamma h_\Gamma^2))}.
\end{equation*}
Noticing that $[\![v_E^\pm]\!]|_\Gamma=0$, by  (\ref{delta_est_yaun2}) and (\ref{delta_est_yaun}), it holds
\begin{equation*}
\begin{aligned}
&\left\|[\![v_E^\pm]\!]\right\|^2_{L^2(U(\Gamma,C_\Gamma h_\Gamma^2))}\leq Ch_\Gamma^4\sum_{s=\pm}|v_E^s|^2_{H^1(U(\Gamma,C_\Gamma h_\Gamma^2))},\\
&\left\|[\![\nabla v_E^\pm]\!]\right\|^2_{L^2(U(\Gamma,C_\Gamma h_\Gamma^2))}\leq Ch_\Gamma^2\sum_{s=\pm}\left(|v_E^s|^2_{H^1(U(\Gamma,\delta_0))}+|v_E^s|^2_{H^2(U(\Gamma,\delta_0) }\right).
\end{aligned}
\end{equation*}
Combining the above inequalities and  the extension result (\ref{extension}) we arrive at 
\begin{equation}\label{est_BK}
\left|v-E^{BK}_hv\right|^2_{H^m(\cup\Omega_h^\pm)}\leq Ch_\Gamma^{4-2m}\|v\|^2_{H^2(\cup\Omega^\pm)},~~m=0,1.
\end{equation}
This completes the proof of this theorem.
\end{proof}

Now we give the proof of Theorem ~\ref{theo_T_inter}.
\begin{proof}[Proof of Theorem ~\ref{theo_T_inter}]
As shown in (\ref{outline_2}), we only need to estimate the second term $({\rm I})_2$.
Combining Theorem~\ref{lem_est_IFEbasis} and  Lemmas \ref{lem_arxi} and \ref{lem_decomp} yields 
\begin{equation}\label{pro_zong_abc}
|\Pi_T v_E^\pm-\left(\Pi_T^{\rm IFE}v\right)^\pm|^2_{H^m(T)}\leq Ch_T^{N-2m}\left(a_T^2+\sum_{i=1}^{N+1}g_{i,T}^2\right)+Ch_T^{N+2-2m}\left(b_T^2+\sum_{i=1}^{N-1}c_{i,T}^2\right).
\end{equation}
Here the constants $a_T$, $b_{T}$, $c_{i,T}$ and $g_{i,T}$ are defined in (\ref{coeff_abc}). We estimate these constants one by one.

\textbf{Derive bounds for $a_T$.} Using the triangle inequality,  the estimate (\ref{tpts}), and the fact that $[\![\Pi_T v_E^\pm]\!]\in\mathbb{P}_1(T)$ can be viewed as a polynomial defined on $\mathbb{R}^N$, we have
\begin{equation}\label{xts_dif_0}
\begin{aligned}
|a_T|&=|[\![\Pi_T v_E^\pm]\!](\mathbf{x}_T^P)|\\
&\leq \left|[\![\Pi_T v_E^\pm]\!](\mathbf{x}_T^S)\right|+\left|[\![\Pi_T v_E^\pm]\!](\mathbf{x}_T^S)-[\![\Pi_T v_E^\pm]\!](\mathbf{x}_T^P) \right|\\
&\leq  \left|[\![\Pi_T v_E^\pm]\!](\mathbf{x}_T^S)\right|+\left|[\![\nabla \Pi_T v_E^\pm]\!]\right|\left|\mathbf{x}_T^S-\mathbf{x}_T^P\right|\\
&\leq  \left|[\![\Pi_T v_E^\pm]\!](\mathbf{x}_T^S)\right|+Ch_T^2\left|[\![\nabla \Pi_T v_E^\pm]\!]\right|.
\end{aligned}
\end{equation}
Since $\mathbf{x}_T^S\in\Gamma_{T}\subset T$, it holds $[\![v_E^\pm]\!](\mathbf{x}_T^S)=0$, and by (\ref{pro_PI_2}), we get 
\begin{equation}\label{xts_dif}
\begin{aligned}
 \left|[\![\Pi_T v_E^\pm]\!](\mathbf{x}_T^S)\right|
 &= \left|[\![\Pi_T v_E^\pm-v_E^\pm]\!](\mathbf{x}_T^S)\right|\\
 &\leq \|[\![\Pi_T v_E^\pm-v_E^\pm]\!]\|_{L^\infty(T)} \\
 &\leq \sum_{s=\pm}\|\Pi_T v_E^s-v_E^s\|_{L^\infty(T)}\\
 &\leq Ch_T^{2-N/2}\sum_{s=\pm}|v_E^s|_{H^2(T)}.
\end{aligned}
\end{equation}
On the other hand,
\begin{equation*}
\begin{aligned}
\left|[\![\nabla \Pi_T v_E^\pm]\!]\right|&=|T|^{-1/2}\left\|[\![\nabla \Pi_T v_E^\pm]\!]\right\|_{L^2(T)}\\
&\leq Ch_T^{-N/2}\sum_{s=\pm} | \Pi_Tv_E^s|_{H^1(T)}\\
&\leq Ch_T^{-N/2}\sum_{s=\pm}|v_E^s|_{H^1(T)},
\end{aligned}
\end{equation*}
where we have used (\ref{pro_PI_3}) in the last inequality.
Combining the above inequalities yields 
\begin{equation}\label{pro_at}
a_T^2\leq Ch_T^{4-N}\sum_{s=\pm}\left(|v_E^s|^2_{H^1(T)}+|v_E^s|^2_{H^2(T)}\right).
\end{equation}

\textbf{Derive bounds for $b_{T}$.} By the fact $[\![\beta_T^\pm \nabla(\Pi_T v_E^\pm)\cdot \mathbf{n}_h]\!]$ is a constant on $T$, we have
\begin{equation*}
\begin{aligned}
|b_{T}|&=|[\![\beta_T^\pm \nabla(\Pi_T v_E^\pm)\cdot \mathbf{n}_h]\!]|\\
&\leq Ch_T^{-N/2}\|[\![\beta_T^\pm \nabla(\Pi_T v_E^\pm)\cdot \mathbf{n}_h]\!]\|_{L^2(T)}\\
&\leq Ch_T^{-N/2}\|[\![\beta_E^\pm \nabla(\Pi_T v_E^\pm)\cdot \mathbf{n}_h]\!]\|_{L^2(T)}+Ch_T^{-N/2}\|[\![(\beta_T^\pm-\beta_E^\pm) \nabla(\Pi_T v_E^\pm)\cdot \mathbf{n}_h]\!]\|_{L^2(T)}.
\end{aligned}
\end{equation*}
For the first term, by (\ref{n_h_esti}) and (\ref{pro_PI_1}), we can derive
\begin{equation*}
\begin{aligned}
 \|[\![\beta_E^\pm& \nabla (\Pi_h v_E^\pm )\cdot \textbf{n}_h]\!]\|_{L^2(T)}= \|[\![\beta_E^\pm \nabla (\Pi_h v_E^\pm -v_E^\pm )\cdot \textbf{n}_h+\beta_E^\pm \nabla v_E^\pm \cdot (\textbf{n}_h-\textbf{n}+\textbf{n})]\!]\|_{L^2(T)}\\
 &\leq \|[\![\beta_E^\pm \nabla (\Pi_h v_E^\pm -v_E^\pm )\cdot \textbf{n}_h]\!]\|_{L^2(T)}+\|\textbf{n}-\textbf{n}_h\|_{L^\infty(T)}\|[\![\beta_E^\pm \nabla v_E^\pm ]\!]\|_{L^2(T)}+ \|[\![ \beta_E^\pm\nabla v_E^\pm\cdot \textbf{n}]\!]\|_{L^2(T)}\\
 &\leq Ch_T\sum_{s=\pm}\left(|v_E^s|_{H^2(T)}+|v_E^s|_{H^1(T)}\right)+\|[\![ \beta_E^\pm\nabla v_E^\pm\cdot \textbf{n}]\!]\|_{L^2(T)}.
 \end{aligned}
\end{equation*}
Similarly, for the second term, it follows from (\ref{betaT}) and (\ref{pro_PI_3}) that 
\begin{equation*}
\begin{aligned}
 \|[\![(\beta_T^\pm-\beta_E^\pm)\nabla (\Pi_h v_E^\pm )\cdot \textbf{n}_h]\!]\|_{L^2(T)}&\leq \sum_{s=\pm}\|\beta_T^s-\beta_E^s\|_{L^\infty(T)}\|\nabla (\Pi_h v_E^s )\cdot \textbf{n}_h\|_{L^2(T)}\\
&\leq Ch_T\sum_{s=\pm}|\Pi_h v_E^s|_{H^1(T)}\\
&\leq Ch_T\sum_{s=\pm}|v_E^s|_{H^1(T)}.
\end{aligned}
\end{equation*}
Combining the above inequalities yields
\begin{equation}\label{pro_b1}
b_T^2\leq  Ch_T^{2-N}\sum_{s=\pm}\left(|v_E^s|^2_{H^2(T)}+|v_E^s|^2_{H^1(T)}\right)+Ch_T^{-N}\|[\![ \beta^\pm\nabla v_E^\pm\cdot \textbf{n}]\!]\|^2_{L^2(T)}.
\end{equation}

\textbf{Derive bounds for $c_{i,T}$.}
Using the property (\ref{sur_grad_h_ineq})  of  tangential gradients and the fact that $[\![ \nabla_{\Gamma_h} (\Pi_h v_E^\pm )]\!]$ is a constant vector on $T$, we have 
\begin{equation*}
\begin{aligned}
|c_{i,T}|&=\left|[\![ \nabla (\Pi_h v_E^\pm )\cdot \textbf{t}_{i,h}]\!]\right|\\
&\leq \left|[\![ \nabla_{\Gamma_h} (\Pi_h v_E^\pm )]\!]\right|\\
&\leq Ch_T^{-N/2}\left\|[\![ \nabla_{\Gamma_h} (\Pi_h v_E^\pm )]\!]\right\|_{L^2(T)}.
\end{aligned}
\end{equation*}
The triangle inequality gives 
\begin{equation*}
\begin{aligned}
\left\|[\![ \nabla_{\Gamma_h} (\Pi_h v_E^\pm )]\!]\right\|_{L^2(T)}&\leq \left\|[\![ \nabla_{\Gamma_h} (\Pi_h v_E^\pm-v_E^\pm )+(\nabla_{\Gamma_h}-\nabla_{\Gamma}+\nabla_{\Gamma})v_E^\pm]\!]\right\|_{L^2(T)}\\
&\leq \left\|[\![ \nabla_{\Gamma_h} (\Pi_h v_E^\pm-v_E^\pm )]\!]\right\|_{L^2(T)}+\left\|[\![\nabla_{\Gamma_h}v_E^\pm-\nabla_{\Gamma}v_E^\pm]\!]\right\|_{L^2(T)}+\left\|[\![\nabla_{\Gamma}v_E^\pm]\!]\right\|_{L^2(T)}.
\end{aligned}
\end{equation*}
Using the property of tangential gradients: $|\nabla_{\Gamma_h}\cdot |\leq |\nabla \cdot|$, we have
\begin{equation*}
\begin{aligned}
\left\|[\![ \nabla_{\Gamma_h} (\Pi_h v_E^\pm-v_E^\pm )]\!]\right\|_{L^2(T)}&\leq \left\|[\![ \nabla (\Pi_h v_E^\pm-v_E^\pm ) ]\!]\right\|_{L^2(T)}\\
&\leq Ch_T\sum_{s=\pm}|v_E^s|_{H^2(T)}.
\end{aligned}
\end{equation*}
By the definition (\ref{def_sur_grad}) and the inequality (\ref{n_h_esti}), we get
\begin{equation*}
\begin{aligned}
\left\|[\![\nabla_{\Gamma_h}v_E^\pm-\nabla_{\Gamma}v_E^\pm]\!]\right\|_{L^2(T)}&=\left\|[\![(\mathbf{n}\cdot \nabla v_E^\pm)\mathbf{n}-(\mathbf{n}_h\cdot \nabla v_E^\pm)\mathbf{n}_h]\!]\right\|_{L^2(T)}\\
&=\left\|[\![(\mathbf{n}\cdot \nabla v_E^\pm)(\mathbf{n}-\mathbf{n}_h)+((\mathbf{n}-\mathbf{n}_h)\cdot \nabla v_E^\pm)\mathbf{n}_h]\!]\right\|_{L^2(T)}\\
&\leq 2\|\mathbf{n}-\mathbf{n}_h\|_{L^\infty(T)}\left\|[\![\nabla v_E^\pm]\!]\right\|_{L^2(T)}\\
&\leq Ch_T\sum_{s=\pm}|v_E^s|_{H^1(T)}.
\end{aligned}
\end{equation*}
Collecting the above inequalities yields 
\begin{equation}\label{est_b2}
c_{i,T}^2\leq Ch_T^{2-N}\sum_{s=\pm}\left(|v_E^s|^2_{H^1(T)}+|v_E^s|^2_{H^2(T)}\right)+Ch_T^{-N}\left\|[\![\nabla_{\Gamma}v_E^\pm]\!]\right\|_{L^2(T)}.
\end{equation}

\textbf{Derive bounds for $g_{i,T}$.}
By the definitions (\ref{def_Nit}), (\ref{def_BK}) and (\ref{def_BK2}), there hold
\begin{equation*}
\mathcal{N}_{i,T}(\Pi_T^{BK}v)=|F_i|^{-1}\sum_{s=\pm}\int_{F_{i}\cap \partial T_h^s}\Pi_Tv_E^s~~\mbox{ and }~~\mathcal{N}_{i,T}(E^{BK}_hv)=|F_i|^{-1}\sum_{s=\pm}\int_{F_{i}\cap \partial T_h^s}v_E^s.
\end{equation*}
Now we can estimate $c_{i,T}$ as
\begin{equation*}
\begin{aligned}
|g_{i,T}|&=\left|\mathcal{N}_{i,T}(\Pi_T^{BK}v-E^{BK}_hv)\right|\\
&=|F_i|^{-1}\left|\sum_{s=\pm}\int_{F_{i}\cap \partial T_h^s}(\Pi_Tv_E^s-v_E^s)\right|\\
&\leq |F_i|^{-1}\sum_{s=\pm}\int_{F_{i}}\left|\Pi_Tv_E^s-v_E^s\right|\\
&\leq|F_i|^{-1/2} \sum_{s=\pm}\left|\Pi_Tv_E^s-v_E^s\right|_{L^2(F_i)}\\
&\leq |F_i|^{-1/2} \sum_{s=\pm}\left(Ch_T^{-1/2}\left\|\Pi_Tv_E^s-v_E^s\right\|_{L^2(T)} +Ch_T^{1/2}\left|\Pi_Tv_E^s-v_E^s\right|_{H^1(T)} \right)\\
&\leq C|F_i|^{-1/2} h_T^{3/2}\sum_{s=\pm}|v_E^s|_{H^2(T)},
\end{aligned}
\end{equation*}
where we have used the Cauchy-Schwarz inequality, the standard trace inequality, and the  estimate  (\ref{pro_PI_1}). Since we assume the triangulation is shape-regular,  we have $|F_i|\geq Ch_T^{N-1}$, and then,
\begin{equation}\label{est_c}
\begin{aligned}
g_{i,T}^2\leq Ch_T^{4-N}\sum_{s=\pm}|v_E^s|^2_{H^2(T)}.
\end{aligned}
\end{equation}

Substituting the estimates (\ref{pro_at})-(\ref{est_c}) into (\ref{pro_zong_abc}), we obtain 
\begin{equation*}
\begin{aligned}
|\Pi_T v_E^\pm-\left(\Pi_T^{\rm IFE}v\right)^\pm|^2_{H^m(T)}&\leq Ch_T^{4-2m}\sum_{s=\pm}\left(|v_E^s|^2_{H^1(T)}+|v_E^s|^2_{H^2(T)}\right)\\
&+Ch_T^{2-2m}\left(\left\|[\![\beta_E^\pm\nabla v_E^\pm\cdot\mathbf{n}]\!]\right\|^2_{L^2(T)}+\left\|[\![\nabla_\Gamma v_E^\pm]\!]\right\|^2_{L^2(T)}\right),
\end{aligned}
\end{equation*}
which together with the standard estimate 
$$
 |v_E^\pm -\Pi_Tv_E^\pm |^2_{H^m(T)}\leq Ch^{4-2m}|v^\pm|^2_{H^2(T)}
$$
yields the desired result (\ref{T_inter}). 
\end{proof}

\section{The finite element method and analysis}\label{sec_IFEM}
\subsection{The method}\label{sec_method}
For each $F\in \mathcal{F}^\Gamma_h$, let $T_1^F$ and $T_2^F$ be two elements sharing the common face $F$. Define the space
\begin{equation*}
\mathbf{Q}_F=\{\mathbf{q}\in L^2(\Omega)^N:~ \mathbf{q}|_{T_i^F}\in \nabla S_h(T_i^F), i=1,2,~ \mathbf{q}|_{\Omega\backslash(T^F_1\cup T^F_2)}=\mathbf{0}\},
\end{equation*}
where $\nabla S_h(T)=\{\nabla v_h:  v_h\in S_h(T) \}.$
The local  {\em  lifting operator}  $\mathbf{r}_F : L^2(F)\rightarrow \mathbf{Q}_F$ is defined by
\begin{equation}\label{def_lift}
\int_{T_1^F\cup T_2^F} \beta^{BK} \mathbf{r}_F(v)\cdot \mathbf{q}=\int_F\{\beta^{BK} \mathbf{q}\cdot \textbf{n}_F\}_Fv \qquad \forall \mathbf{q}\in \mathbf{Q}_F,
\end{equation}
where $\beta^{BK}(\mathbf{x}):=E_h^{BK}\beta(\mathbf{x})$ and $\{\cdot\}_F$ stands for the average over $F$, i.e., $\{v\}_F=(v|_{T^F_1}+v|_{T^F_2})/2$.  Define the IFE space $V_{h,0}^{\rm IFE}=\{v\in V_{h}^{\rm IFE} :  \int_F v=0~ \forall F\in\mathcal{F}_h^b\}$ and the following bilinear forms:
\begin{equation}\label{def_AAH}
\begin{aligned}
&a_h(z,v):=\int_\Omega \beta^{BK} \nabla_h z\cdot\nabla_h v,\\
&b_h(z,v):=-\sum_{F\in\mathcal{F}_h^\Gamma}\int_F\left(\{\beta^{BK}\nabla_h z\cdot\textbf{n}_F\}_F[v]_F+\{\beta^{BK}\nabla_h v\cdot\textbf{n}_F\}_F[z]_F\right),\\
&s_h(z,v):=8\sum_{F\in\mathcal{F}_h^\Gamma}\int_{T_1^F\cup T_2^F}\beta^{BK}\mathbf{r}_F([z]_F)\cdot \mathbf{r}_F([v]_F),\\
&A_h(z,v):=a_h(z,v)+b_h(z,v)+s_h(z,v),
\end{aligned}
\end{equation}
where $(\nabla_h v)|_{T}=\nabla v|_T$ for all $T\in \mathcal{T}_h^{non}$ and $(\nabla_h v)|_{T_h^\pm}=\nabla v|_{T_h^\pm}$ for all $T\in \mathcal{T}_h^{\Gamma}$.
The immersed finite method reads: find $u_h\in V_{h,0}^{\rm IFE}$ such that
\begin{equation}\label{method2}
A_h(u_h,v_h)=\int_{\Omega}f^{BK}v_h \qquad\forall v_h\in V_{h,0}^{\rm IFE},
\end{equation}
where $f^{BK}:=E_h^{BK}f$. Recalling the definition of $E_h^{BK}$ in (\ref{def_BK}), the function $f^{BK}$ relies on the extensions $f_E^+$ and $f_E^-$. Since $f^\pm\in L^2(\Omega^\pm)$, we can simply use the trivial extension of $f^\pm$ to satisfy (\ref{extension}) (i.e., $f_E^\pm=0$ outside $\Omega^\pm$).

Clearly, the method is symmetric and does not require a manually chosen stabilization parameter. We note that the term $b_h(\cdot,\cdot)$ is crucial to ensure the optimal convergence (see \cite{2021ji_nonconform}). Comparing with the traditional Crouzeix-Raviart finite element method, the additional terms $b_h(\cdot,\cdot)$ and $s_h(\cdot,\cdot)$ are only evaluated on interface faces, and thus the extra computational cost is not significant in general. We also note that we do not need to solve linear systems for the lifting $\mathbf{r}_F(v)$ for a given function $v\in L^2(F)$.  Using the fact that $\nabla S_h(T)=\mbox{ span}\{\boldsymbol{\eta},~ \mathbf{t}_{i,h}, i=1,..., N-1\}$, where $\boldsymbol{\eta}|_{T_h^\pm}=\beta^\mp_T\mathbf{n}_h$, the lifting $\mathbf{r}_F(v)$  has an explicit formula (see \cite{2021ji_IFE}).

\subsection{Continuity and coercivity}
Define the mesh-dependent norms $\|\cdot\|_h$ and $\interleave \cdot \interleave_h$ by
\begin{equation}\label{def_nor2}
\begin{aligned}
\|v\|^2_{h}&=a_h(v,v),\\
\interleave v \interleave_h^2&=\|v\|_{h}^2+\sum_{F\in\mathcal{F}_h^\Gamma}\left(h_F\|\{\beta^{BK}\nabla_h v\}_F\|^2_{L^2(F)}+h_F^{-1}\| [v]_F\|^2_{L^2(F)}\right)+s_h(v,v),
\end{aligned}
\end{equation}
where $h_F$ denotes the diameter of $F$. Using the Poincar\'e-Friedrichs  inequalities for piecewise $H^1$ functions (see \cite{brenner2003poincare}), we have 
\begin{equation}\label{poin_ineq}
\|v\|^2_{L^2(\Omega)}\leq C\sum_{T\in\mathcal{T}_h}|v_h|^2_{H^1(T)} \leq C \|v\|^2_{h}\qquad\forall v\in \left(H_0^1(\Omega)\cap\widetilde{H}^2(\Omega)\right)+ V_{h,0}^{\rm IFE},
\end{equation}
which implies that $\| \cdot \|_h$ and $\interleave\cdot\interleave_h$ are indeed norms on the space $\left(H_0^1(\Omega)\cap\widetilde{H}^2(\Omega)\right)+ V_{h,0}^{\rm IFE}$.
It follows from  the Cauchy-Schwarz inequality that $A_h(\cdot,\cdot)$ is bounded by $\interleave\cdot\interleave_h$, i.e.,
\begin{equation}\label{conti}
|A_h(z,v)|\leq \interleave z \interleave_h \interleave v\interleave_h.
\end{equation}
The following lemma shows that  $A_h(\cdot,\cdot)$  is coercive on the IFE space $V_{h,0}^{\rm IFE}$ with respect to  $\|\cdot\|_{h}$.
\begin{lemma}\label{lem_Coercivity}
It holds that
\begin{equation}\label{Coercivity}
A_h(v,v)\geq \frac{1}{2}\|v\|_{h}^2\qquad \forall v\in V_{h,0}^{\rm IFE}.
\end{equation}
\end{lemma}
\begin{proof}
For all $v\in V_{h,0}^{\rm IFE}$, choosing $\mathbf{q}|_{T_{i}^F}=\nabla v|_{T_{i}^F}$, $i=1,2$ and  $\mathbf{q}|_{\Omega\backslash(T_{1}^F\cup T_{2}^F) }=\mathbf{0}$ in (\ref{def_lift}) yields
\begin{equation*}
\int_{T_{1}^F\cup T_{2}^F} \beta^{BK} \mathbf{r}_F([v]_F)\cdot \nabla v =\int_F \{ \beta^{BK} \nabla v \cdot \textbf{n}_F \}_F [v]_F.
\end{equation*}
Then we have
\begin{equation*}
\begin{aligned}
b_h(v,v)=-2\sum_{F\in\mathcal{F}_h^\Gamma}\int_F\{\beta^{BK}\nabla v\cdot\textbf{n}_F\}_F[v]_F=-2\sum_{F\in\mathcal{F}_h^\Gamma}\int_{T_{1}^F\cup T_{2}^F} \beta^{BK} \mathbf{r}_F([v]_F)\cdot \nabla v.
\end{aligned}
\end{equation*}
It follows from the Cauchy-Schwarz inequality that 
\begin{equation*}
\begin{aligned}
|b_h(v,v)|\leq \left(4\sum_{F\in\mathcal{F}_h^\Gamma} \int_{T_1^F\cup T_2^F} \beta^{BK} \mathbf{r}_F([v]_F)\cdot \mathbf{r}_F([v]_F)\right)^{1/2}\left(\sum_{F\in\mathcal{F}_h^\Gamma} \int_{T_{1}^F\cup T_{2}^F} \beta^{BK} \nabla v \cdot \nabla v\right)^{1/2}.
\end{aligned}
\end{equation*}
Since each element is calculated at most $N+1$ times, it holds for both $N=2$ and $N=3$ that
\begin{equation*}
\sum_{F\in\mathcal{F}_h^\Gamma} \int_{T_{1}^F\cup T_{2}^F} \beta^{BK} \nabla v \cdot \nabla v \leq 4\sum_{T\in\mathcal{T}_h}\int_T \beta^{BK} \nabla v \cdot \nabla v.
\end{equation*}
Therefore, we have
\begin{equation*}
\begin{aligned}
\left|b_h(v,v)\right|&\leq \left(\frac{1}{2}s_h(v,v)\right)^{1/2}\left(4\sum_{T\in\mathcal{T}_h}\int_T \beta^{BK} \nabla v \cdot \nabla v \right)^{1/2}\\
&\leq s_h(v,v)+\frac{1}{2} \sum_{T\in\mathcal{T}_h}\int_T \beta^{BK} \nabla v \cdot \nabla v,
\end{aligned}
\end{equation*}
 which leads to
\begin{equation*}
a_h(v,v)+b_h(v,v)+s_h(v,v)\geq \frac{1}{2} \sum_{T\in\mathcal{T}_h}\int_T \beta^{BK} \nabla v \cdot \nabla v  = \frac{1}{2}\|v\|_{h}^2.
\end{equation*}
This completes the proof of this lemma.
\end{proof}

\subsection{Norm-equivalence for IFE functions}
We show that  the norms $\|\cdot\|_{h}$ and $ \interleave\cdot\interleave_h$ are equivalent on the IFE space $V_{h,0}^{\rm IFE}$. To this end, we first prove the  trace inequality for IFE functions in the following lemma.
\begin{lemma}[Trace inequality]\label{lema_trace}
There exists a  constant $C$ independent of $h$ and the interface location relative to the mesh such that
\begin{equation}\label{trace_IFE}
\|\nabla v\|_{L^2(\partial T)}\leq Ch_T^{-1/2}\|\nabla v\|_{L^2(T)}~\quad \forall v\in S_h(T)\quad\forall T\in\mathcal{T}_h^\Gamma.
\end{equation}
\end{lemma}
\begin{proof}
By the definition of $S_h(T)$ (see also Remark~\ref{remark_th}), there holds
$$\mathbf{n}_h \cdot \nabla v^+=(\beta_T^-/\beta_T^+)\mathbf{n}_h \cdot \nabla v^-,\qquad \nabla_{\Gamma_{h,T}} v^+=\nabla_{\Gamma_{h,T}} v^-,$$
which together with the decomposition 
$\nabla v^\pm=(\mathbf{n}_h \cdot \nabla v^\pm)\mathbf{n}_h+\nabla_{\Gamma_{h,T}} v^\pm$ (see (\ref{def_sur_grad}))
yields
\begin{equation*}
\min\left\{\beta_T^-/\beta_T^+,1\right\}\|\nabla v^-\|_{L^2(D)} \leq \|\nabla v^+\|_{L^2(D)}\leq \max\left\{\beta_T^-/\beta_T^+,1\right\}\|\nabla v^-\|_{L^2(D)}
\end{equation*}
for any subdomain $D\subset T$. Using the above inequalities we can derive 
\begin{equation*}
\begin{aligned}
\|\nabla v\|^2_{L^2(\partial T)}&\leq \sum_{s=\pm}\|\nabla v^s\|^2_{L^2(\partial T)}\leq \sum_{s=\pm}Ch_T^{-1}\|\nabla v^s\|^2_{L^2(T)}\leq Ch_T^{-1}\|\nabla v\|^2_{L^2(T)},
\end{aligned}
\end{equation*}
which completes the proof of this lemma.
\end{proof}
With the trace inequality we can derive the following stability estimate for the operator $\mathbf{r}_F$.
\begin{lemma}\label{lem_stab_lift}
There exists a  constant $C$ independent of $h$ and the interface location relative to the mesh such that
\begin{equation*}
\|\mathbf{r}_F(v)\|_{L^2(\Omega)}\leq Ch_F^{-1/2}\|v\|_{L^2(F)}\quad \forall v\in L^2(F)\quad \forall F\in\mathcal{F}_h^\Gamma.
\end{equation*}
\end{lemma}
\begin{proof}
Choosing  $\mathbf{q}=\mathbf{r}_F(v)$ in (\ref{def_lift}) yields 
\begin{equation*}
\begin{aligned}
\|\mathbf{r}_F(v)\|^2_{L^2(T_1^F\cup T_2^F)}&\leq C\|v\|_{L^2(F)}(\|\mathbf{r}_F(v)|_{T_1^F}\|_{L^2(F)}+\|\mathbf{r}_F(v)|_{T_2^F}\|_{L^2(F)})\\
&\leq Ch_F^{-1/2}\|v\|_{L^2(F)}\|\mathbf{r}_F(v)\|_{L^2(T_1^F\cup T_2^F)},
\end{aligned}
\end{equation*}
where in the last inequality we have used the trace inequality (\ref{trace_IFE}) since $\mathbf{r}_F(v)|_{T_i^F}\in\nabla S_h(T_i^F)$. Using the above inequality and the fact $\mathbf{r}_F(v)|_{\Omega\backslash(T^F_1\cup T^F_2)}=0$ we completes the proof of this lemma.
\end{proof}

Since $v|_T\in H^1(T)$ and $\int_F[v]_F=0$ for all $v\in V_{h,0}^{\rm IFE}$, we have the following standard result 
\begin{equation*}
\|[v]_F\|_{L^2(F)}^2\leq Ch_F(|v|^2_{H^1(T_1^F)}+|v|^2_{H^1(T_2^F)}) \quad \forall F\in \mathcal{F}_h \quad\forall v\in V_{h,0}^{\rm IFE}.
\end{equation*}
Combining this with the definition (\ref{def_nor2}) and Lemmas \ref{lema_trace} and \ref{lem_stab_lift}, we can easily obtain the norm-equivalence as shown in the following lemma.
\begin{lemma}\label{lema_equ}
There exists a constant $C$ independent of $h$ and the interface location relative to the mesh such that
\begin{equation*}
\|v\|_{h} \leq \interleave v \interleave_h\leq C\|v\|_{h}\qquad \forall v\in V_{h,0}^{\rm IFE}.
\end{equation*}
\end{lemma}
\subsection{Interpolation error estimates in the norm $\interleave\cdot \interleave_h$}
\begin{lemma}\label{ener_app}
Suppose $v\in \widetilde{H}^2(\Omega)$. Let $v^{BK}:=E_h^{BK}v$, then there exists a constant $C$ independent of $h$ and the interface location relative to the mesh such that
\begin{equation*}
\interleave v^{BK}-\Pi_h^{\rm IFE}v \interleave_h\leq Ch\|v\|_{H^2(\cup \Omega^\pm)}.
\end{equation*}
\end{lemma}
\begin{proof}
Using (\ref{outline_1}) and (\ref{IFE_inter_gamma}), we can bound the first term in the norm $\interleave\cdot\interleave_h$ as
\begin{equation*}
\|v^{BK}-\Pi_h^{\rm IFE}v\|_h\leq Ch\|v\|_{H^2(\cup \Omega^\pm)}.
\end{equation*}
For the second term, recalling the definition of $E^{BK}_h$ in (\ref{def_BK}) and using (\ref{IFE_inter_gamma}) again we can derive 
\begin{equation*}
\begin{aligned}
\sum_{F\in\mathcal{F}_h^\Gamma}h_F&\|\{\beta^{BK}\nabla_h (v^{BK}-\Pi_h^{\rm IFE}v)\}_F\|^2_{L^2(F)}\\
&=\sum_{F\in\mathcal{F}_h^\Gamma}\sum_{s=\pm}h_F\|\{\beta^{BK}\nabla (v_E^s-(I_h^{\rm IFE}v)^s)\}_F\|^2_{L^2(F\cap \partial T_h^\pm)}\\
&\leq C\sum_{F\in\mathcal{F}_h^\Gamma}\sum_{s=\pm}h_F\|\{\nabla (v_E^s-(I_h^{\rm IFE}v)^s)\}_F\|^2_{L^2(F)}\\
&\leq C\sum_{T\in\mathcal{T}_h^\Gamma}\sum_{s=\pm}\left(|v_E^s-(I_h^{\rm IFE}v)^s|^2_{H^1(T)}+h_T^2|v_E^s|^2_{H^2(T)}\right)\\
&\leq Ch_\Gamma^2\|v\|^2_{H^2(\cup \Omega^\pm)},
\end{aligned}
\end{equation*}
where in the second inequality we have used the standard trace inequality since $v_E^\pm-(I_h^{\rm IFE}v)^\pm\in H^1(T)$. Analogously, by the standard trace inequality, (\ref{T_inter}), (\ref{zhengti1})-(\ref{zhengti2}) and (\ref{extension}) we have
\begin{equation*}
\begin{aligned}
\sum_{F\in\mathcal{F}_h^\Gamma}h_F^{-1}&\| [v^{BK}-\Pi_h^{\rm IFE}v]_F\|^2_{L^2(F)}\\
&\leq C\sum_{T\in\mathcal{T}_h^\Gamma}\sum_{s=\pm}\left(h_T^{-2}|v_E^s-(I_h^{\rm IFE}v)^s|^2_{L^2(T)}+|v_E^s-(I_h^{\rm IFE}v)^s|^2_{H^1(T)}\right)\\
&\leq C \sum_{T\in\mathcal{T}_h^\Gamma}  \left(\left\|[\![\beta_E^\pm\nabla v_E^\pm\cdot\mathbf{n}]\!]\right\|^2_{L^2(T)}+\left\|[\![\nabla_\Gamma v_E^\pm]\!]\right\|^2_{L^2(T)}+h_T^2\sum_{s=\pm}\|v_E^s\|^2_{H^2(T)}\right)\\
&\leq Ch_\Gamma^2\|v\|^2_{H^2(\cup \Omega^\pm)},
\end{aligned}
\end{equation*}
which together with  Lemma~\ref{lem_stab_lift} leads to
\begin{equation*}
s_h(v^{BK}-\Pi_h^{\rm IFE}v,v^{BK}-\Pi_h^{\rm IFE}v)\leq Ch_\Gamma^2\|v\|^2_{H^2(\cup\Omega^\pm)}.
\end{equation*}
The lemma then follows from the above inequalities and the definition of the norm $\interleave\cdot\interleave_h$.
 \end{proof}
 \subsection{Consistency}
Define 
$\tilde{f}_E^\pm:=-\nabla\cdot\beta_E^\pm\nabla u_E^\pm$ in   $\Omega_{\delta_0}^\pm.$
From the original PDE (\ref{p1.1}), we can see  $\tilde{f}_E^\pm-f_E^\pm=0$ on $\Omega^\pm$, while $\tilde{f}_E^\pm-f_E^\pm$ is not in general equal to zero in $\Omega_h^\pm\backslash\Omega^\pm$.  For simplicity of notation, we let $u^{BK}:=E_h^{BK}u$ and define  $\tilde{f}^{BK}$ such that $\tilde{f}^{BK}|_{\Omega_h^\pm}= \tilde{f}_E^\pm|_{\Omega_h^\pm}$, then it holds $-\nabla_h\cdot (\beta^{BK}\nabla_h u^{BK})=\tilde{f}^{BK}$ in $\Omega$.
Multiplying this by $v_h\in V_{h,0}^{\rm IFE}$ and integrating by parts  yields  
\begin{equation}\label{bk_int}
\int_{\Omega}\beta^{BK}\nabla_h u^{BK} \cdot\nabla_h v_h+\int_{\Gamma_h}[\![\beta_E^\pm\nabla u_E^\pm\cdot\mathbf{n}_h]\!]v_h-\sum_{F\in\mathcal{F}_h}\int_F [\beta^{BK}\nabla_h u^{BK}\cdot \mathbf{n}_F v_h]_F=\int_\Omega \tilde{f}^{BK}v_h.
\end{equation}
Using fact $[u^{BK}]_F=[\beta^{BK}\nabla_h u^{BK}\cdot \mathbf{n}_F]_F=0$ for all $F\in\mathcal{F}_h$, we have the following relations
\begin{equation*}
\begin{aligned}
&[\beta^{BK}\nabla_h u^{BK}\cdot \mathbf{n}_F v_h]_F=\{\beta^{BK}\nabla_h u^{BK}\cdot \mathbf{n}_F\}_F[v_h]_F+\{\beta^{BK}\nabla v_h\cdot \mathbf{n}_F\}_F[u^{BK}]_F,\\
&s_h(u^{BK},v_h)=8\sum_{F\in\mathcal{F}_h^\Gamma}\int_{T_1^F\cup T_2^F}\beta^{BK}\mathbf{r}_F([u^{BK}]_F)\mathbf{r}_F([v_h]_F)=0.
\end{aligned}
\end{equation*} 
Combining these with (\ref{bk_int}), (\ref{def_AAH}) and (\ref{method2}), we obtain 
\begin{equation}\label{consis_001}
\begin{aligned}
&A_h(u^{BK}-u_h,v_h)\\
&\quad=\underbrace{-\int_{\Gamma_h}[\![\beta_E^\pm\nabla u_E^\pm\cdot\mathbf{n}_h]\!]v_h}_{{\rm (II)}_1}+\underbrace{\sum_{F\in\mathcal{F}_h^{non}}\int_F\beta^{BK}\nabla_h u^{BK}\cdot \mathbf{n}_F[v_h]_F}_{{\rm (II)}_2}+\underbrace{\int_{\Omega}(\tilde{f}^{BK}-f^{BK})v_h}_{{\rm (II)}_3}.
\end{aligned}
\end{equation}

\textbf{Derive bounds for ${\rm (II)}_1$.} By the Cauchy-Schwarz inequality, we have
\begin{equation}\label{ii_es}
|{\rm (II)}_1|\leq \left\|[\![\beta_E^\pm \nabla u_E^\pm\cdot \mathbf{n}_h]\!]\right\|_{L^2(\Gamma_h)}\|v_h\|_{L^2(\Gamma_h)}.
\end{equation}
To estimate the terms on the right-hand side of the above inequality, we need the following lemma.
\begin{lemma}\label{lem_gamma_h}
There is a constant $C$ depending only on $\Gamma$ such that
\begin{equation*}
\|v\|^2_{L^2(\Gamma_h)}\leq C\|v\|^2_{L^2(\Gamma)}+Ch^2_\Gamma  \|\nabla v\|^2_{L^2(U(\Gamma,C_\Gamma h^2_\Gamma))}\qquad\forall  v\in H^1(U(\Gamma,C_\Gamma h^2_\Gamma)).
\end{equation*}
\end{lemma}
\begin{proof}
See (A.4)-(A.6) in \cite{burman2018Acut}.
\end{proof}
With this lemma we can derive the estimate for $ \left\|[\![\beta_E^\pm \nabla u_E^\pm\cdot \mathbf{n}_h]\!]\right\|_{L^2(\Gamma_h)}$.
\begin{lemma}\label{lem_flux_app}
There is a constant $C$ independent of $h$ and the interface location relative to the mesh such that
\begin{equation}\label{flux_app}
\left\|[\![\beta_E^\pm \nabla v_E^\pm\cdot \mathbf{n}_h]\!]\right\|^2_{L^2(\Gamma_h)}\leq Ch_\Gamma^2\sum_{s=\pm}\left( | v_E^s |^2_{H^1(U(\Gamma,\delta_0))}+| v_E^s |^2_{H^2(U(\Gamma,\delta_0))}\right)\qquad \forall v\in \widetilde{H}^2(\Omega).
\end{equation}
\end{lemma}
\begin{proof}
The triangle inequality gives 
\begin{equation}\label{pro_gammah_app}
\left\|[\![\beta_E^\pm \nabla v_E^\pm\cdot \mathbf{n}_h]\!]\right\|^2_{L^2(\Gamma_h)}\leq 2\left\|[\![\beta_E^\pm \nabla v_E^\pm\cdot (\mathbf{n}_h-\mathbf{n})]\!]\right\|^2_{L^2(\Gamma_h)}+2\left\|[\![\beta_E^\pm \nabla v_E^\pm\cdot \mathbf{n}]\!]\right\|^2_{L^2(\Gamma_h)}.
\end{equation}
By (\ref{n_h_esti}), Lemma~\ref{lem_gamma_h} and the inequalities (\ref{beta_ext1})-(\ref{beta_ext2}) for $\beta_E^\pm$, the first term can be estimated as 
\begin{equation}\label{pro_gammah_app1}
\begin{aligned}
\left\|[\![\beta_E^\pm \nabla v_E^\pm\cdot (\mathbf{n}_h-\mathbf{n})]\!]\right\|^2_{L^2(\Gamma_h)}&\leq Ch^2_\Gamma \left\|[\![\beta_E^\pm \nabla v_E^\pm]\!]\right\|^2_{L^2(\Gamma_h)}\\
&\leq Ch_\Gamma^2 \|[\![\beta_E^\pm \nabla v_E^\pm]\!]\|^2_{L^2(\Gamma)}+Ch^4_\Gamma  \| [\![\beta_E^\pm \nabla v_E^\pm]\!] \|^2_{H^1(U(\Gamma,\delta_0))}\\
&\leq Ch_\Gamma^2\sum_{s=\pm}  \|\nabla v_E^s\|^2_{L^2(\Gamma)}+Ch_\Gamma^4\sum_{s=\pm}\sum_{i=1,2} | v_E^s |^2_{H^i(U(\Gamma,\delta_0))}\\
&\leq Ch_\Gamma^2\sum_{s=\pm}\left( | v_E^s |^2_{H^1(U(\Gamma,\delta_0))}+| v_E^s |^2_{H^2(U(\Gamma,\delta_0))}\right),
\end{aligned}
\end{equation}
where in the last inequality we have applied the global trace inequality on the domain  $U^s(\Gamma,\delta_0)$ for estimating $\|\nabla v_E^s\|_{L^2(\Gamma)}$.
For the second term on the right-hand side of (\ref{pro_gammah_app}), applying Lemma~\ref{lem_gamma_h} again and using the fact $[\![\beta_E^\pm \nabla v_E^\pm\cdot \mathbf{n}]\!]|_\Gamma=0$, we have 
\begin{equation}\label{pro_gammah_app2}
\begin{aligned}
\left\|[\![\beta_E^\pm \nabla v_E^\pm\cdot \mathbf{n}]\!]\right\|^2_{L^2(\Gamma_h)}&\leq Ch^2_\Gamma  \|\nabla [\![\beta_E^\pm \nabla v_E^\pm\cdot \mathbf{n}]\!]\|^2_{L^2(U(\Gamma,\delta_0))}\\
&\leq Ch_\Gamma^2\sum_{s=\pm}\left( | v_E^s |^2_{H^1(U(\Gamma,\delta_0))}+| v_E^s |^2_{H^2(U(\Gamma,\delta_0))}\right).
\end{aligned}
\end{equation}
Substituting  (\ref{pro_gammah_app1}) and (\ref{pro_gammah_app2}) into (\ref{pro_gammah_app}) yields the desired result. 
\end{proof}
To estimate the term $\|v_h\|_{L^2(\Gamma_h)}$ in (\ref{ii_es}),  we first need the inverse inequality for the IFE functions as shown in the following lemma.

\begin{lemma}[Inverse inequality]\label{lema_inverse}
There exists a  constant $C$ independent of $h$ and the interface location relative to the mesh such that
\begin{equation}\label{inverse_IFE}
\|\nabla \phi\|_{L^2( T)}\leq Ch_T^{-1}\| \phi\|_{L^2(T)}~\quad \forall \phi\in S_h(T)\quad\forall T\in\mathcal{T}_h^\Gamma.
\end{equation}
\end{lemma}
\begin{proof}
 \begin{figure} [htbp]
\centering
\includegraphics[width=0.5\textwidth]{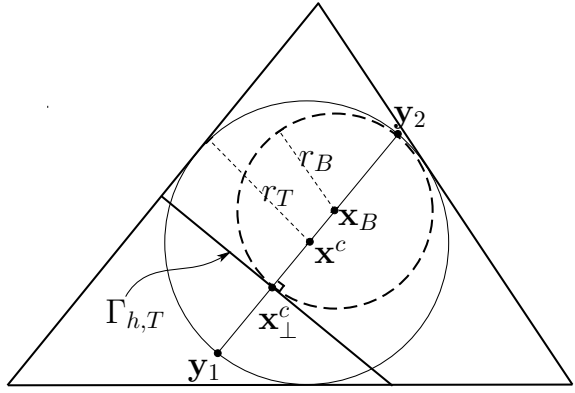}
 \caption{Construction of the ball $B$ (dash line) for the 2D case\label{fig_inver}} %% label for entire figure
\end{figure}

Let $\mathbf{x}_\perp=\mathbf{p}_{\Gamma_h^{ext}}(\mathbf{x})$. Using the interface conditions in  the definition of $S_h(T)$ (see also Remark~\ref{remark_th}), we have 
\begin{equation*}
\begin{aligned}
\phi^+(\mathbf{x})-\phi^-(\mathbf{x})&=\nabla (\phi^+-\phi^-)\cdot\mathbf{n}_h  (\mathbf{x}-\mathbf{x}_\perp)\cdot\mathbf{n}_h\\ 
&=
\left\{\begin{array}{l}
(\beta^-_T/\beta_T^+-1)\nabla \phi^-\cdot\mathbf{n}_h  (\mathbf{x}-\mathbf{x}_\perp)\cdot\mathbf{n}_h,\\
(1-\beta^+_T/\beta_T^-)\nabla \phi^+\cdot\mathbf{n}_h  (\mathbf{x}-\mathbf{x}_\perp)\cdot\mathbf{n}_h,
\end{array}
\right.
\end{aligned}
\end{equation*}
which leads to
\begin{equation}\label{pro_inve1}
\|\phi^+\|^2_{L^2(T_h^-)}\leq 2\|\phi^-\|^2_{L^2(T_h^-)}+Ch_T^{2+N}|B|^{-1}|\phi^{s_0}|^2_{H^1(B)},
\end{equation}
where the superscript $s_0$ and  the ball $B$ are chosen as follows. Let $B_T$ be the largest ball inscribed in $T$ with the center $\mathbf{x}^c$ and the radius $r_T$. Let $\mathbf{x}^c_\perp=\mathbf{p}_{\Gamma_h^{ext}}(\mathbf{x}^c)$. The line $\mathbf{x}^c\mathbf{x}^c_\perp$ intersects  $\partial B_T$ at points $\mathbf{y}_1$ and $\mathbf{y}_2$ such that  $|\mathbf{y}_2-\mathbf{x}^c_\perp|\geq |\mathbf{y}_1-\mathbf{x}^c_\perp|$. The superscript $s_0=+$ or $-$ is chosen such that $\mathbf{x}^c\in  \overline{T_h^{s_0}}$. If  $\Gamma_{h,T}\cap B_T=\emptyset$, we choose $B=B_T$, otherwise,
$B$ is the ball centered at $\mathbf{x}_B=(\mathbf{x}^c_\perp+\mathbf{y}_2)/2$ with the radius $r_B=|\mathbf{x}^c_\perp-\mathbf{y}_2|/2$; see Figure~\ref{fig_inver} for an illustration for the 2D case. 
 It is easy to verify that, for both cases, the ball $B\subset T_h^{s_0}$ and its radius $r_B=\min\{r_T, (r_T+|\mathbf{x}^c-\mathbf{x}^c_\perp|)/2\}\geq r_T/2$, thus, $|B|\geq Ch_T^N$. Applying the standard inverse inequality for $v^{s_0}$ on the ball $B$, the inequality (\ref{pro_inve1}) becomes 
\begin{equation*}
\begin{aligned}
\|\phi^+\|^2_{L^2(T_h^-)}&\leq 2\|\phi^-\|^2_{L^2(T_h^-)}+C\|\phi^{s_0}\|^2_{L^2(B)}\\
&\leq 2\|\phi^-\|^2_{L^2(T_h^-)}+C\|\phi\|^2_{L^2(T)}.
\end{aligned}
\end{equation*}
Analogously, 
\begin{equation*}
\|\phi^-\|^2_{L^2(T_h^+)}\leq 2\|\phi^+\|^2_{L^2(T_h^+)}+C\|\phi\|^2_{L^2(T)}.
\end{equation*}
Using the above inequalities we can derive 
\begin{equation*}
\begin{aligned}
\|\nabla \phi\|^2_{L^2( T)}&\leq \sum_{s=\pm} \|\nabla \phi^s\|^2_{L^2( T)}\leq \sum_{s=\pm}Ch_T^{-2}\| \phi^s\|^2_{L^2(T)}\leq Ch_T^{-2}\| \phi\|^2_{L^2(T)},
\end{aligned}
\end{equation*}
which completes the proof of this lemma.
\end{proof}

 The following lemma shows the relations between  the IFE function and its  Crouzeix-Raviart interpolant.
\begin{lemma}\label{lem_rela_PI_IFE}
For any $\phi\in S_h(T)$ with $T\in\mathcal{T}_h^\Gamma$, there exist positive constant $c$  and $C$ independent of $h$ and the interface location relative to the mesh such that 
\begin{align}
&c|\phi|_{H^m(T)}\leq |\Pi_T\phi|_{H^m(T)}\leq C |\phi|_{H^m(T)}, ~~m=0,1,\label{rela_PI_1}\\
& \|\phi-\Pi_T\phi\|_{L^2(T)}\leq Ch_T|\phi|_{H^1(T)}.\label{rela_PI_2}
\end{align}
\end{lemma}
\begin{proof}
From (\ref{pro_rela_ife_cr}) we known $\phi=\Pi_T\phi+\alpha\phi_{J}$ with
\begin{equation*}
\alpha=\frac{(\beta^-_T/\beta^+_T-1)\nabla\Pi_T\phi \cdot \mathbf{n}_h}{1+(\beta^-_T/\beta^+_T-1)|T_h^+|/|T|} ~\mbox{ and } ~  \phi_{J}=w-\Pi_Tw.
\end{equation*}
From (\ref{pro_basi_fenmu}) we have 
$$|\alpha|\leq C|\nabla\Pi_T\phi |.$$
Similar to the estimate for $\Upsilon_T$ in Lemma~\ref{lem_arxi}, we can prove 
\begin{equation*}
|\phi_J|_{W^m_\infty(T)}\leq Ch_T^{1-m}.
\end{equation*}
Therefore, we obtain
\begin{equation*}
\begin{aligned}
|\alpha\phi_{J}|_{H^m(T)}&\leq C|\nabla\Pi_T\phi |h_T^{1-m}h_T^{N/2}\\
&\leq Ch_T^{1-m}|\Pi_T\phi |_{H^1(T)}\\
&\leq \left\{\begin{array}{l}
 C|\Pi_T\phi |_{H^m(T)},\\
 Ch_T^{1-m}|\phi |_{H^1(T)} \leq C|\phi |_{H^m(T)},
\end{array}\right.
\end{aligned}
\end{equation*}
where we have used the standard inverse inequality for $\Pi_T\phi$, the stability result (\ref{pro_PI_3}) and the inverse inequality (\ref{inverse_IFE}) for IFE functions.  The lemma follows directly from the above inequalities and the relation $\phi=\Pi_T\phi+\alpha\phi_{J}$.
\end{proof}

We also need a connection operator which maps a standard Crouzeix-Raviart finite element function to a function in $H^1(\Omega)$.
Let $V^{con}_h$ be the $P_2$ Lagrange finite element space associated with $\mathcal{T}_h$ for $N=2$ and the $P_3$ Lagrange finite element space for $N=3$.  The connection operator $R_h: V_h\rightarrow V^{con}_h$ was defined in \cite{brenner2003poincare}. 
Let $\Xi(T)=\{T^\prime \in\mathcal{T}_h : \partial T\cap\partial T^\prime \not=\emptyset \}$. Under the assumption that the triangulation is shape-regular, we have the following  properties of the  operator $R_h$. There exist constants $c$ and $C$ such that 
\begin{align}
&c\sum_{T\in\mathcal{T}_h}|v|^2_{H^i(T)}\leq \sum_{T\in\mathcal{T}_h}|R_hv|^2_{H^i(T)} \leq C \sum_{T\in\mathcal{T}_h}|v|^2_{H^i(T)},\quad i=0,1, ~~ \forall v\in V_h,  \label{R1}\\
&\|R_hv-v\|^2_{L^2(T)}\leq \sum_{T^\prime \in \Xi(T)}Ch_{T^\prime}^2|v|^2_{H^1(T^\prime)}\qquad \forall v\in V_h,\label{R2}
\end{align}
where  the first property is from Corollary 3.3 in \cite{brenner2003poincare} and the second property is (3.7) in \cite{brenner2003poincare}.

Now we can derive the bound for the term  $\|v_h\|_{L^2(\Gamma_h)}$.
\begin{lemma}\label{lem_vh_gamma}
There exists a  constant $C$ independent of $h$ and the interface location relative to the mesh such that
\begin{equation}\label{vh_gamma}
\|v_h\|_{L^2(\Gamma_h)}\leq C\|v_h\|_h\qquad \forall v_h\in V_{h,0}^{\rm IFE}.
\end{equation}
\end{lemma}
\begin{proof} 
We have the split
\begin{equation*}
\|v_h\|_{L^2(\Gamma_h)}\leq \| R_h \Pi_h v_h\|_{L^2(\Gamma_h)}+\| v_h- \Pi_h v_h\|_{L^2(\Gamma_h)}+\| \Pi_h v_h-R_h \Pi_h v_h\|_{L^2(\Gamma_h)}.
\end{equation*}
Here we emphasize that we used the standard Crouzeix-Raviart interpolation operator $\Pi_h$  in the above inequality, not the IFE interpolation operator $\Pi_h^{\rm IFE}$.
Since $R_h \Pi_h v_h\in H^1(\Omega)$, it follows from  Lemma~\ref{lem_gamma_h} and the global trace inequality that
\begin{equation*}
\begin{aligned}
 \| R_h \Pi_h v_h\|^2_{L^2(\Gamma_h)}&\leq C\|R_h\Pi_h v_h\|^2_{H^1(\Omega)}\leq C\sum_{T\in\mathcal{T}_h}\|\Pi_h v_h\|^2_{H^1(T)}\\
 &\leq C\sum_{T\in\mathcal{T}_h}\|v_h\|^2_{H^1(T)}\leq C\sum_{T\in\mathcal{T}_h}|v_h|^2_{H^1(T)},
\end{aligned}
\end{equation*}
where we have used (\ref{R1}) in the second inequality, (\ref{rela_PI_1}) in the third inequality, and (\ref{poin_ineq}) in the last inequality.
By the well-known trace inequality on the interface (see, e.g., \cite{hansbo2002unfitted,2016High}), we get
\begin{equation*}
\begin{aligned}
\| v_h- \Pi_h v_h\|^2_{L^2(\Gamma_h)}&\leq \sum_{T\in\mathcal{T}_h^\Gamma}C(h_T^{-1}\| v_h- \Pi_h v_h\|^2_{L^2(T)}+h_T| v_h- \Pi_h v_h|^2_{H^1(T)})\\
&\leq \sum_{T\in\mathcal{T}_h^\Gamma}C h_T| v_h|^2_{H^1(T)},
\end{aligned}
\end{equation*}
where we used (\ref{rela_PI_1}) and (\ref{rela_PI_2}) in the last inequality.
Applying the well-known trace inequality on the interface again gives 
\begin{equation*}
\begin{aligned}
\|  \Pi_hv_h-R_h \Pi_h v_h\|^2_{L^2(\Gamma_h)}&\leq \sum_{T\in\mathcal{T}_h^\Gamma}C(h_T^{-1}\| \Pi_hv_h-R_h \Pi_h v_h\|^2_{L^2(T)}+h_T| \Pi_hv_h-R_h \Pi_h v_h|^2_{H^1(T)})\\
&\leq \sum_{T\in\mathcal{T}_h^\Gamma}Ch_T^{-1}\| \Pi_hv_h-R_h \Pi_h v_h\|^2_{L^2(T)}\\
&\leq \sum_{T\in\mathcal{T}_h^\Gamma}\sum_{T^\prime\in \Xi(T)}Ch_T^{-1}h^2_{T^\prime}| \Pi_hv_h|^2_{H^1(T^\prime)}\\&\leq \sum_{T\in\mathcal{T}_h}Ch| v_h|^2_{H^1(T)},
\end{aligned}
\end{equation*}
where in the second inequality we used the standard inverse inequality, in the third inequality we used the estimate (\ref{R2}), and in the last inequality we used (\ref{rela_PI_1}). Collecting the above inequalities yields the desired result.
\end{proof}
Substituting (\ref{flux_app}) and (\ref{vh_gamma}) into (\ref{ii_es})  and using the extension result (\ref{extension}) we obtain
\begin{equation}\label{consis_fen1}
|{\rm (II)}_1|\leq  Ch_\Gamma\|u\|_{H^2(\cup \Omega^\pm)}\|v_h\|_h.
\end{equation}

\textbf{Derive bounds for ${\rm (II)}_2$.} It suffices to consider the case $F\in\mathcal{F}_h^{non}$ with $F\subset \partial T$, $T\in\mathcal{T}_h^\Gamma$.  Suppose $F\subset\Omega^{s_0}$ with $s_0=+$ or $-$, then we have the standard result from the nonconforming finite element analysis  
\begin{equation*}
\begin{aligned}
\left|\int_F\beta^{BK}\nabla u^{BK}\cdot \mathbf{n}_F[v_h]_F\right|&=\left|\int_F\beta_E^{s_0}\nabla u_E^{s_0}\cdot \mathbf{n}_F[v_h]_F\right|\\
&\leq Ch|u_E^{s_0}|_{H^2(T)}\left(|v_h|^2_{H^1(T_1^{F})}+|v_h|^2_{H^1(T_1^{F})}\right)^{1/2},
\end{aligned}
\end{equation*}
which together with an analogous estimate for other faces gives 
\begin{equation}\label{consis_fen2}
|{\rm (II)}_2|\leq Ch\sum_{s=\pm}|u_E^s|_{H^2(\Omega_{\delta_0}^\pm)}\|v_h\|_h\leq Ch\|u\|_{H^2(\cup\Omega^\pm)}\|v_h\|_h.
\end{equation}

\textbf{Derive bounds for ${\rm (II)}_3$.}  
By definition, we have
\begin{equation*}
\tilde{f}^{BK}-f^{BK}=
\left\{
\begin{aligned}
&\tilde{f}_E^\pm-f_E^\pm  &&\mbox{in } \Omega_h^\pm\backslash\Omega^\pm,\\
&0&&\mbox{otherwise.}
\end{aligned}\right.
\end{equation*}
By (\ref{extension}) and (\ref{regular}), it holds
\begin{equation*}
\begin{aligned}
\|\tilde{f}_E^\pm\|_{L^2(\Omega_{\delta_0}^\pm)}&=\|\nabla\cdot \beta_E^\pm\nabla u_E^\pm\|_{L^2(\Omega_{\delta_0}^\pm)}\leq  C \| u_E^\pm\|_{H^2(\Omega_{\delta_0}^\pm)}\\
&\leq  C \| u^\pm\|_{H^2(\Omega^\pm)}\leq  C\|f\|_{L^2(\Omega)}.
\end{aligned}
\end{equation*}
Recalling $\Omega^\triangle=(\Omega_h^-\cap\Omega^+)\cup(\Omega_h^+\cap\Omega^-)$, we can derive
\begin{equation}\label{pro_II2_0}
|{\rm (II)}_3|=\left|\int_{\Omega}(\tilde{f}^{BK}-f^{BK})v_h\right|\leq C\|f\|_{L^2(\Omega)}\|v_h\|_{L^2(\Omega^\triangle)},
\end{equation}

\begin{lemma}\label{lem_mis_IFE}
There exists a  constant $C$ independent of $h$ and the interface location relative to the mesh such that
\begin{equation*}
\|v_h\|_{L^2(\Omega^\triangle)} \leq Ch\|v_h\|_h \qquad \forall v_h\in V_{h,0}^{\rm IFE}.
\end{equation*}
\end{lemma}
\begin{proof}
By (\ref{tdelta_U}) and the triangle inequality we have 
\begin{equation*}
\begin{aligned}
\|v_h\|_{L^2(\Omega^\triangle)} &\leq \|v_h\|_{L^2(U(\Gamma,C_\Gamma h_\Gamma^2))}\\
&\leq \|R_h\Pi_h v_h\|_{L^2(U(\Gamma,C_\Gamma h_\Gamma^2))}+\|v_h-\Pi_hv_h\|_{L^2(\Omega)}+\|\Pi_hv_h-R_h\Pi_hv_h\|_{L^2(\Omega)}.
\end{aligned}
\end{equation*}
Using (\ref{delta_est_yaun}), (\ref{rela_PI_1})-(\ref{R2}) and (\ref{poin_ineq}) we obtain
\begin{equation*}
\begin{aligned}
\|v_h\|_{L^2(\Omega^\triangle)} &\leq Ch_\Gamma \|R_h\Pi_h v_h\|_{H^1(\Omega)}+Ch\|v_h\|_h+Ch\|\Pi_hv\|_h\\
&\leq Ch\|v_h\|_h,
\end{aligned}
\end{equation*}
which completes the proof.
\end{proof}
It follows from the above lemma and (\ref{pro_II2_0}) that 
\begin{equation}\label{consis_fen3}
|{\rm (II)}_3|\leq Ch\|f\|_{L^2(\Omega)}\|v_h\|_{h}.
\end{equation}

Substituting (\ref{consis_fen1}), (\ref{consis_fen2}) and (\ref{consis_fen3})  into (\ref{consis_001}) yields the following lemma.
\begin{lemma}\label{lem_consis}
Let $u$ and $u_h$ be the solutions of  problem (\ref{p1.1})-(\ref{p1.5}) and problem (\ref{method2}), respectively. Then it holds for all $v_h\in V_{h,0}^{\rm IFE}$ that
\begin{equation*}
\begin{aligned}
&\left|A_h(u^{BK}-u_h,v_h)\right|\leq Ch(\|u\|_{H^2(\cup\Omega^\pm)}+\|f\|_{L^2(\Omega)})\|v_h\|_{h}.
\end{aligned}
\end{equation*}
\end{lemma}

\subsection{Error estimates}
With these preparations, we are ready to derive the $H^1$ error estimate for the proposed IFE method. 
 \begin{theorem}\label{theo_mainH1}
Let $u$ and $u_h$ be the solutions of problem (\ref{p1.1})-(\ref{p1.5}) and problem (\ref{method2}), respectively.  Then there exists a constant $C$ independent of $h$ and the interface location relative to the mesh such that
\begin{equation}\label{H_1_error_method2}
\interleave u^{BK}-u_h \interleave_h\leq Ch(\|u\|_{H^2(\cup\Omega^\pm)}+\|f\|_{L^2(\Omega)}),
\end{equation}
where $u^{BK}=E_h^{BK}u$.
\end{theorem}
\begin{proof}
The triangle inequality gives 
\begin{equation}\label{pro_lll_tri}
\interleave u^{BK}-u_h \interleave_h\leq\interleave u^{BK}-\Pi_h^{\rm IFE}u \interleave_h+\interleave \Pi_h^{\rm IFE}u-u_h \interleave_h.
\end{equation}
For simplicity of notation, let $e_h:=\Pi_h^{\rm IFE}u-u_h$. From Lemmas~\ref{lem_Coercivity} and \ref{lema_equ}, we have
\begin{equation*}
\begin{aligned}
\interleave e_h \interleave^2_h&\leq CA_h(\Pi_h^{\rm IFE}u-u_h,e_h)\\
&\leq CA_h(\Pi_h^{\rm IFE}u-u^{BK},e_h)+CA_h(u^{BK}-u_h,e_h).
\end{aligned}
\end{equation*}
By the continuity (\ref{conti}) and Lemma~\ref{lem_consis} we further have
\begin{equation*}
\interleave \Pi_h^{\rm IFE}u-u_h \interleave_h \leq C\interleave \Pi_h^{\rm IFE}u-u^{BK}\interleave_h+Ch(\|u\|_{H^2(\cup\Omega^\pm)}+\|f\|_{L^2(\Omega)}).
\end{equation*}
Substituting this into (\ref{pro_lll_tri}) and using Lemma~\ref{ener_app} yields the desired result.
\end{proof}
\begin{remark}\label{remark_h1_err}
We also have the following error estimate for the exact solution 
\begin{equation*}
\| u-u_h\|_h\leq Ch(\|u\|_{H^2(\cup\Omega^\pm)}+\|f\|_{L^2(\Omega)}),
\end{equation*}
which is obtained by using the triangle inequality $\| u-u_h\|_h\leq \| u-u^{BK}\|_h+\|u^{BK}-u_h\|_h$ and the estimates (\ref{est_BK}) and (\ref{H_1_error_method2}).
\end{remark}

\subsection{Condition number analysis}
With the help of the inverse inequality (\ref{inverse_IFE}) and the relation (\ref{rela_PI_1})  we can obtain the following theorem showing that  the condition number of the stiffness matrix of the proposed IFE method has the usual bound $O(h^{-2})$ with the hidden constant independent of the  interface location relative to the mesh.
\begin{lemma}
Let $\{\phi_i : i=1,...,N_{J}\}$ be the basis for  $V_{h,0}^{\rm IFE}$ and ${\rm \mathbf{A}}$ be the stiffness matrix defined by
${\rm \mathbf{A}}(i,j)=A_h(\phi_i,\phi_j)~~\forall i,j=1,...,N_{J}$. Suppose the family of triangulations is also quasi-uniform, i.e.,  there is a constant $C$ such that $h_T^{-1}\leq Ch^{-1}$ for any $T\in \mathcal{T}_h$ and any triangulation $ \mathcal{T}_h$.  Then the $l_2$ condition number, ${\rm cond}_2({\rm \mathbf{A}})$, of ${\rm \mathbf{A}}$ is bounded by 
\begin{equation*}
{\rm cond}_2({\rm \mathbf{A}})\leq Ch^{-2},
\end{equation*}
where the constant $C$ is independent of $h$ and the interface location relative to the mesh.
\end{lemma}
\begin{proof}
For any vector $\mathbf{v}\in\mathbb{R}^{N_J}$, there is a function $v_h\in V_{h,0}^{\rm IFE}$ such that $v_h=\sum_{i=1}^{N_J}\mathbf{v}(i)\phi_i$. Noticing that $\Pi_h v_h$ is the corresponding function belonging to the standard Crouzeix-Raviart finite element space $V_h$, we have the following standard result
\begin{equation*}
ch^{-N}\|\Pi_h v_h\|^2_{L^2(\Omega)}\leq  |\mathbf{v}|^2=\mathbf{v}^T\mathbf{v}\leq Ch^{-N}\|\Pi_h v_h\|^2_{L^2(\Omega)},
\end{equation*}
where $c$ and $C$ are general constants.
From the inequality $h_T^{-1}\leq Ch^{-1}$ (the quasi-uniform assumption) and the inverse inequality (\ref{inverse_IFE}) for IFE functions, it holds  
\begin{equation*}
 \|v_h\|^2_h \leq C\sum_{T\in\mathcal{T}_h}\|\nabla v_h \|^2_{L^2( T)}\leq Ch^{-2}\| v_h\|^2_{L^2(\Omega)}.
\end{equation*}
Therefore, using the above inequalities we have
\begin{equation*}
\begin{aligned}
\Lambda_{{\rm max}}({\rm \mathbf{A}} )&=\max_{\mathbf{v}\in \mathbb{R}^{N_J}}\frac{ \mathbf{v}^T {\rm \mathbf{A}} \mathbf{v} }{ \mathbf{v}^T\mathbf{v}}\leq \max_{v_h\in V_{h,0}^{\rm IFE}}\frac{ A_h(v_h,v_h) }{ ch^{-N}\|\Pi_hv_h\|^2_{L^2(\Omega)}}\\
&\leq \max_{v_h\in V_{h,0}^{\rm IFE}}\frac{ C\interleave v_h \interleave^2_h }{ h^{-N}\|v_h\|^2_{L^2(\Omega)}}\leq \max_{v_h\in V_{h,0}^{\rm IFE}}\frac{ C\| v_h \|^2_h }{ h^{-N}\|v_h\|^2_{L^2(\Omega)}}\\
&\leq \max_{v_h\in V_{h,0}^{\rm IFE}}\frac{ Ch^{-2} \|v_h\|^2_{L^2(\Omega)}  }{ h^{-N}\|v_h\|^2_{L^2(\Omega)}} \leq Ch^{N-2},
\end{aligned}
\end{equation*}
where we have used (\ref{rela_PI_1}) and (\ref{conti}) in the second inequality  and  Lemma~\ref{lema_equ} in the third inequality.
Analogously, we can derive 
\begin{equation*}
\begin{aligned}
\Lambda_{{\rm min}}({\rm \mathbf{A}} )&=\min_{\mathbf{v}\in \mathbb{R}^{N_J}}\frac{ \mathbf{v}^T {\rm \mathbf{A}} \mathbf{v} }{ \mathbf{v}^T\mathbf{v}}\geq \min_{v_h\in V_{h,0}^{\rm IFE}}\frac{ A_h(v_h,v_h) }{ Ch^{-N}\|\Pi_hv_h\|^2_{L^2(\Omega)}}\\
&\geq \min_{v_h\in V_{h,0}^{\rm IFE}}\frac{C \| v_h \|^2_h }{ h^{-N}\|v_h\|^2_{L^2(\Omega)}}\geq Ch^{N},
\end{aligned}
\end{equation*}
where we have used (\ref{Coercivity}) in the second inequality.
Combining the above estimates yields the desired result 
\begin{equation*}
{\rm cond}_2({\rm \mathbf{A}})= \frac{\Lambda_{{\rm max}}({\rm \mathbf{A}} )}{\Lambda_{{\rm min}}({\rm \mathbf{A}} )}\leq Ch^{-2}.
\end{equation*} 
\end{proof}
 
\section{Extension to anisotropic interface problems}\label{sec_exten}
In this section we consider the anisotropic interface problem, i.e., the coefficient $\beta(\mathbf{x})$ is replaced by a discontinuous tensor-valued function $\mathbb{B}(\mathbf{x})$. For simplicity, we consider a piecewise constant tensor, i.e., $\mathbb{B}|_{\Omega^\pm}=\mathbb{B}^\pm$, $\mathbb{B}^\pm\in\mathbb{R}^{N\times N}$. We assume there exist constants $b_M^\pm$ and $b_m^\pm$ such that $b_M^\pm\geq \mathbf{y}^T\mathbb{B}^\pm \mathbf{y}\geq b_m^\pm>0$ for all $\mathbf{y}\in \mathbb{R}^N$ with $\mathbf{y}^T\mathbf{y}=1$.
The extension of our IFE method to this case is obvious. Next, we show that the analysis can also be extended to this case easily.

Throughout our previous analysis, it is no hard to see that the key is the unisolvence of IFE basis functions and the estimate (\ref{pro_basi_fenmu}). 
In the following we show that these results also hold for tensor-valued coefficients. 
On each interface element $T$, the local IFE space now is 
\begin{equation*}
S_h(T):=\{\phi\in L^2(T) : \phi|_{T_h^\pm}\in  \mathbb{P}_1(T_h^\pm),~ [\phi]_{\Gamma_{h,T}}=0,~ [\mathbb{B}_T\nabla \phi\cdot \mathbf{n}_h]_{\Gamma_{h,T}}=0 \},
\end{equation*}
where $\mathbb{B}_T$ is a piecewise constant tensor defined by $\mathbb{B}_T|_{T_h^\pm}=\mathbb{B}^\pm$.
Substituting (\ref{phi_decomp}) into  the jump condition $[\mathbb{B}_T\nabla \phi\cdot \mathbf{n}_h]_{\Gamma_{h,T}}=0$ we have
\begin{equation*}
[\mathbb{B}_T\nabla \phi_J\cdot \mathbf{n}_h]_{\Gamma_{h,T}}\alpha=-[\mathbb{B}_T\nabla \phi_0\cdot \mathbf{n}_h]_{\Gamma_{h,T}}.
\end{equation*}
It is clear that 
\begin{equation*}
\mathbb{B}_T\nabla \phi_J\cdot \mathbf{n}_h=\mathbf{n}_h^T\mathbb{B}_T\mathbf{n}_h(\nabla \phi_J\cdot \mathbf{n}_h)+\sum_{i=1}^{N-1}\mathbf{n}_h^T\mathbb{B}_T\mathbf{t}_{i,h}(\nabla \phi_J\cdot \mathbf{t}_{i,h}).
\end{equation*}
By (\ref{def_phiJ}) and (\ref{phiJ_pro}), we have
\begin{equation*}
\begin{aligned}
[\mathbf{n}_h^T\mathbb{B}_T\mathbf{n}_h(\nabla \phi_J\cdot \mathbf{n}_h)]_{\Gamma_{h,T}}
&=\mathbf{n}_h^T\mathbb{B}_T^+\mathbf{n}_h(\nabla \phi_J^-\cdot \mathbf{n}_h+1)-\mathbf{n}_h^T\mathbb{B}_T^-\mathbf{n}_h(\nabla \phi_J^-\cdot \mathbf{n}_h)\\
&=\mathbf{n}_h^T\mathbb{B}_T^+\mathbf{n}_h(1+(1-\rho) \nabla  \phi_J^-\cdot \mathbf{n}_h  )\\
&=\mathbf{n}_h^T\mathbb{B}_T^+\mathbf{n}_h(1+(\rho-1) \nabla  \Pi_T w\cdot \mathbf{n}_h  ),
\end{aligned}
\end{equation*}
where $\rho:=(\mathbf{n}_h^T\mathbb{B}_T^-\mathbf{n}_h)/(\mathbf{n}_h^T\mathbb{B}_T^+\mathbf{n}_h)\geq b_m^-/b_M^+>0$.
By (\ref{phiJ_pro}), we also have 
$$\nabla \phi_J^+\cdot \mathbf{t}_{i,h}=\nabla \phi_J^-\cdot \mathbf{t}_{i,h}=\nabla \Pi_T w\cdot \mathbf{t}_{i,h}.$$ 
Then we obtain
\begin{equation*}
[\mathbb{B}_T\nabla \phi_J\cdot \mathbf{n}_h]_{\Gamma_{h,T}}=\mathbf{n}_h^T\mathbb{B}_T^+\mathbf{n}_h(1+(\rho-1) \nabla  \Pi_T w\cdot \mathbf{n}_h  )+\sum_{i=1}^{N-1}[\mathbf{n}_h^T\mathbb{B}_T\mathbf{t}_{i,h}]_{\Gamma_{h,T}}(\nabla \Pi_T w\cdot \mathbf{t}_{i,h}).
\end{equation*}
By replacing $\mathbf{n}_h$ by $\mathbf{t}_{i,h}$ in the proof of Lemma~\ref{lem_01}, we find 
$\nabla\cdot(w\mathbf{t}_{i,h})|_{T_h^+}=0$, and thus we can prove that $\nabla \Pi_T w\cdot \mathbf{t}_{i,h}=0$. Collecting above results, we get an equation similar to (\ref{eq_mu}),
\begin{equation*}
\left(1+(\rho-1)\nabla \Pi_T w\cdot \mathbf{n}_h\right)\alpha=-(\mathbf{n}_h^T\mathbb{B}_T^+\mathbf{n}_h)^{-1}[\mathbb{B}_T\nabla \phi_0\cdot \mathbf{n}_h]_{\Gamma_{h,T}}.
\end{equation*}
Similarly to Theorem~\ref{theo_basis}, we can use Lemma~\ref{lem_01} to show that the IFE basis functions for this case are also unisolvent on arbitrary triangles/tetrahedrons regardless of the interface.

The remaining analysis of the IFE space and method can be easily adapted to this case if the regularity result (\ref{regular}) holds. For example, in the proof of Lemma~\ref{lem_arxi} the construction of $\Upsilon_T(\mathbf{x})$ should be changed to 
\begin{equation*}
\Upsilon_{T}(\mathbf{x})=z(\mathbf{x})-\widetilde{\Pi}^{\rm IFE}_Tz(\mathbf{x})\quad\mbox{ with }\quad
z(\mathbf{x})=\left\{
\begin{aligned}
&(\mathbf{n}_h^T\mathbb{B}_T^+\mathbf{n}_h)^{-1}(\mathbf{x}-\mathbf{x}_T^P)\cdot\mathbf{n}_h\quad &&\mbox{ if } \mathbf{x}\in T_h^+,\\
&0&&\mbox{ if } \mathbf{x}\in T_h^-.
\end{aligned} 
\right.
\end{equation*}

\begin{remark}
The unisolvence of IFE basis functions for anisotropic interface problems in both 2D and 3D is another advantage of using integral-values as degrees of freedom. It should be noted that the authors in \cite{An2014A} give some counter examples to show that the IFE basis functions based on nodal-value degrees of freedom may not exist even on isosceles right triangles and for SPD tensors. 
\end{remark}

\section{Numerical examples}\label{lem_num}
In this section we present some numerical examples for the proposed IFE method in 3D. The computational domain is $\Omega=(-1,1)^3$. The interface is the zero level set of a given function $\varphi(x,y,z)$ so that $\Omega^+=\{(x,y,z)\in \mathbb{R}^3 : \varphi(x,y,z)>0\}$ and $\Omega^-=\{(x,y,z)\in \mathbb{R}^3 : \varphi(x,y,z)<0\}$. The exact solution is $u|_{\Omega^\pm}=u^\pm$ with given $u^+$ and $u^-$. We use uniform meshes of the domain $\Omega$, consisting of $M\times M\times M$ equally sized cubes. Each of these cubes is then subdivided into six tetrahedrons. In all examples, the discrete interface is chosen as $\Gamma_h=\{(x,y,z)\in\mathbb{R}^3 : I_h\varphi=0\}$, and the $L^2$ and $H^1$ errors are computed via
\begin{equation*}
\|e_h\|_{L^2}:=\left(\sum_{s=\pm}\|u^s-u_h\|^2_{L^2(\Omega_h^s)}\right)^{1/2},\quad|e_h|_{H^1}:=\left(\sum_{s=\pm}\|\nabla_h u^s-\nabla_h  u_h\|^2_{L^2(\Omega_h^s)} \right)^{1/2}.
\end{equation*}
We use the explicit formula (\ref{express_IFE_basis}) to compute the IFE basis functions in the code. 

\textbf{Example 1.} The coefficient $\beta(x,y,z)$ is a piecewise constant, i.e., $\beta|_{\Omega^\pm}=\beta^\pm$. The functions $\varphi$, $u^+$ and $u^-$ are chosen as
\begin{equation*}
 \begin{aligned}
 &\varphi(x,y,z) = \sqrt{x^2+y^2+z^2}-r_0,\\
 &u^+(x,y,z)= (x^2+y^2+z^2)^{3/2}/\beta^++(1/\beta^--1/\beta^+)r_0^3,\\
 &u^-(x,y,z)= (x^2+y^2+z^2)^{3/2}/\beta^-,
 \end{aligned}
\end{equation*}
where $r_0=\pi/6.28$. We test the example with the coefficient $\beta$ ranging from small to large jumps: $\beta^+=2,\beta^-=1$; $\beta^+=1000,\beta^-=1$; $\beta^+=1,\beta^-=1000$. The errors and orders of convergence are shown in Tables \ref{ex1_1}-\ref{ex1_3}. The condition numbers and orders are shown in Table~\ref{ex_cond}. These numerical results indicate that the proposed IFE method achieves the optimal convergence and the condition number of the stiffness matrix has the usual bound $O(h^{-2})$, which are in agreement with our theoretical analysis.

\begin{table}[H]
\caption{ Numerical results for \textbf{Example 1} with  $\beta^+=2$, $\beta^-=1$\label{ex1_1}}
\begin{center}
\begin{tabular}{c|c c|c c}
\hline
       $M$   &     $\|e_h\|_{L^2}$      &  Order  &     $|e_h|_{H^1}$      & Order  \\   \hline
         5  &   3.605E-02 &        &      2.780E-01  &           \\ \hline
        10  &   9.498E-03 &  1.92  &  1.250E-01  & 1.15   \\ \hline
        20  &   2.403E-03 &  1.98  &  5.993E-02  & 1.06   \\ \hline
        40  &   6.037E-04 &  1.99  &  2.919E-02  & 1.04   \\ \hline
        80  &   1.510E-04 &  2.00  &  1.439E-02  & 1.02   \\ \hline
\end{tabular}
\end{center}
\end{table}

\begin{table}[H]
\caption{ Numerical results for \textbf{Example 1} with $\beta^+=1000$, $\beta^-=1$\label{ex1_2}}
\begin{center}
\begin{tabular}{c|c c|c c}
\hline
       $M$   &     $\|e_h\|_{L^2}$      &  Order  &     $|e_h|_{H^1}$      & Order  \\   \hline 
         5  &    3.204E-02 &            &    8.778E-02 &         \\   \hline 
        10  &   1.065E-02 &  1.59   &    5.194E-02 &  0.76   \\   \hline 
        20  &   2.828E-03 &  1.91   &    2.802E-02 &  0.89   \\   \hline 
        40  &   6.983E-04 &  2.02   &    1.103E-02 &  1.34   \\   \hline 
        80  &   1.727E-04 &  2.02   &    4.573E-03 &  1.27   \\   \hline 
\end{tabular}
\end{center}
\end{table}

\begin{table}[H]
\caption{ Numerical results for \textbf{Example 1} with $\beta^+=1$, $\beta^-=1000$\label{ex1_3}}
\begin{center}
\begin{tabular}{c|c c|c c}
\hline
       $M$   &     $\|e_h\|_{L^2}$      &  Order  &     $|e_h|_{H^1}$      & Order  \\   \hline 
         5  &    7.759E-02 &            &    5.953E-01 &         \\   \hline
        10  &   1.998E-02 &  1.96   &    2.553E-01 &  1.22   \\   \hline
        20  &   4.429E-03 &  2.17   &    1.143E-01 &  1.16   \\   \hline
        40  &   1.099E-03 &  2.01   &    5.628E-02 &  1.02   \\   \hline
        80  &   2.721E-04 &  2.01   &    2.788E-02 &  1.01   \\   \hline
\end{tabular}
\end{center}
\end{table}

\begin{table}[H]
\caption{ Condition numbers for  \textbf{Example 1} (denoted by \textbf{Ex1}) and \textbf{Example 2} (denoted by \textbf{Ex2}) \label{ex_cond}}
\begin{center}
\begin{tabular}{c|c c|c c|c c|c c}
\hline
 & \multicolumn{2}{c|}{\textbf{Ex1}: $\beta^+/\beta^-=2$}& \multicolumn{2}{c|}{\textbf{Ex1}: $\beta^+/\beta^-=10^3$}&\multicolumn{2}{c|}{\textbf{Ex1}: $\beta^+/\beta^-=10^{-3}$}&\multicolumn{2}{c}{\textbf{Ex2}}\\ \hline
     $M$  &      ${\rm cond}_2({\rm \mathbf{A}})$         &   Order   &    ${\rm cond}_2({\rm \mathbf{A}})$            &  Order     &      ${\rm cond}_2({\rm \mathbf{A}})$         &   Order    &     ${\rm cond}_2({\rm \mathbf{A}})$       &   Order          \\ \hline
         5   &  1.057E+02   &            &   5.792E+04    &             &     4.215E+05   &           &      1.700E+02   &                    \\ \hline
        10  &  4.346E+02   &  -2.04   &  4.179E+05    & -2.85     &    1.336E+06   &  -1.66    &   6.911E+02   &  -2.02   \\ \hline
        20  &  1.753E+03   &  -2.01   &  1.629E+06    & -1.96     &    9.202E+06   &  -2.78    &   2.782E+03   &  -2.01   \\ \hline
       40  &    7.027E+03  &  -2.00   &  7.574E+06    & -2.22     &    3.956E+07   &   -2.10       &   1.116E+04   &  -2.00   \\ \hline
\end{tabular}
\end{center}
\end{table}

\textbf{Example 2} (Variable coefficient). The functions $\varphi$, $\beta^\pm$ and $u^\pm$ are chosen as
\begin{equation*}
\begin{aligned}
&\varphi(x,y,z)=\frac{x^2}{a^2}+\frac{y^2}{b^2}+\frac{x^2}{c^2}-1,\\
&\beta^+(x,y,z)=\sin (x+y+z)+2,\\
&\beta^-(x,y,z)=\cos (x+y+z)+2,\\
&u^\pm(x,y,z)=\varphi/\beta^\pm,
\end{aligned}
\end{equation*}
where $a=0.3$, $b=0.5$, $c=0.6$. It is easy to verify that the jump conditions (\ref{p1.2}) and (\ref{p1.3}) are satisfied.

For this interface problem with variable coefficients, in the construction of IFE basis function on an interface element $T$, we simply choose $\beta_T^\pm=\beta^\pm(x_i,y_i,z_i)$ with $(x_i,y_i,z_i)$ being an arbitrary vertex of the element $T$ to satisfy (\ref{betaT}). Numerical results are reported in Tables~\ref{ex2} and \ref{ex_cond}, which show the optimal convergence of the IFE method and the usual bound $O(h^{-2})$ of the condition number.

\begin{table}[H]
\caption{Numerical results for \textbf{Example 2} \label{ex2}}
\begin{center}
\begin{tabular}{c|c c|c c}
\hline
       $M$   &     $\|e_h\|_{L^2}$      &  Order  &     $|e_h|_{H^1}$      & Order  \\   \hline 
          5  &   2.674E-01 &           &     2.336E+00 &         \\   \hline 
        10  &   6.684E-02 &  2.00   &    1.195E+00 &  0.97   \\   \hline 
        20  &   1.642E-02 &  2.03   &    5.989E-01 &  1.00   \\   \hline 
        40  &   4.148E-03 &  1.99   &    2.993E-01 &  1.00   \\   \hline 
        80  &   1.030E-03 &  2.01   &    1.495E-01 &  1.00  \\   \hline 
\end{tabular}
\end{center}
\end{table}

\textbf{Example 3} (Sliver experiment). In this example we investigate the dependence of the condition numbers on small-cut elements and the contrast $\beta^+/\beta^-$. We deliberately create small-cut elements by setting $M=10$ and defining $\varphi(x,y,z)=x_0$ with $x_0$ varying from $0$ to $2/M=0.2$. 

We plot $\log_{10}( {\rm cond}_2({\rm \mathbf{A}}) )$ versus $x_0$ and $\log_{10}(\beta^+/\beta^-)$ in Figure~\ref{Fig_ex3}.  From the numerical results, we can observe that the condition number is not sensitive to the small-cut elements and grows linearly with respective to $\max(\beta^+,\beta^-)/\min(\beta^+,\beta^-)$.

 \begin{figure} [htbp]
\centering
\includegraphics[width=0.45\textwidth]{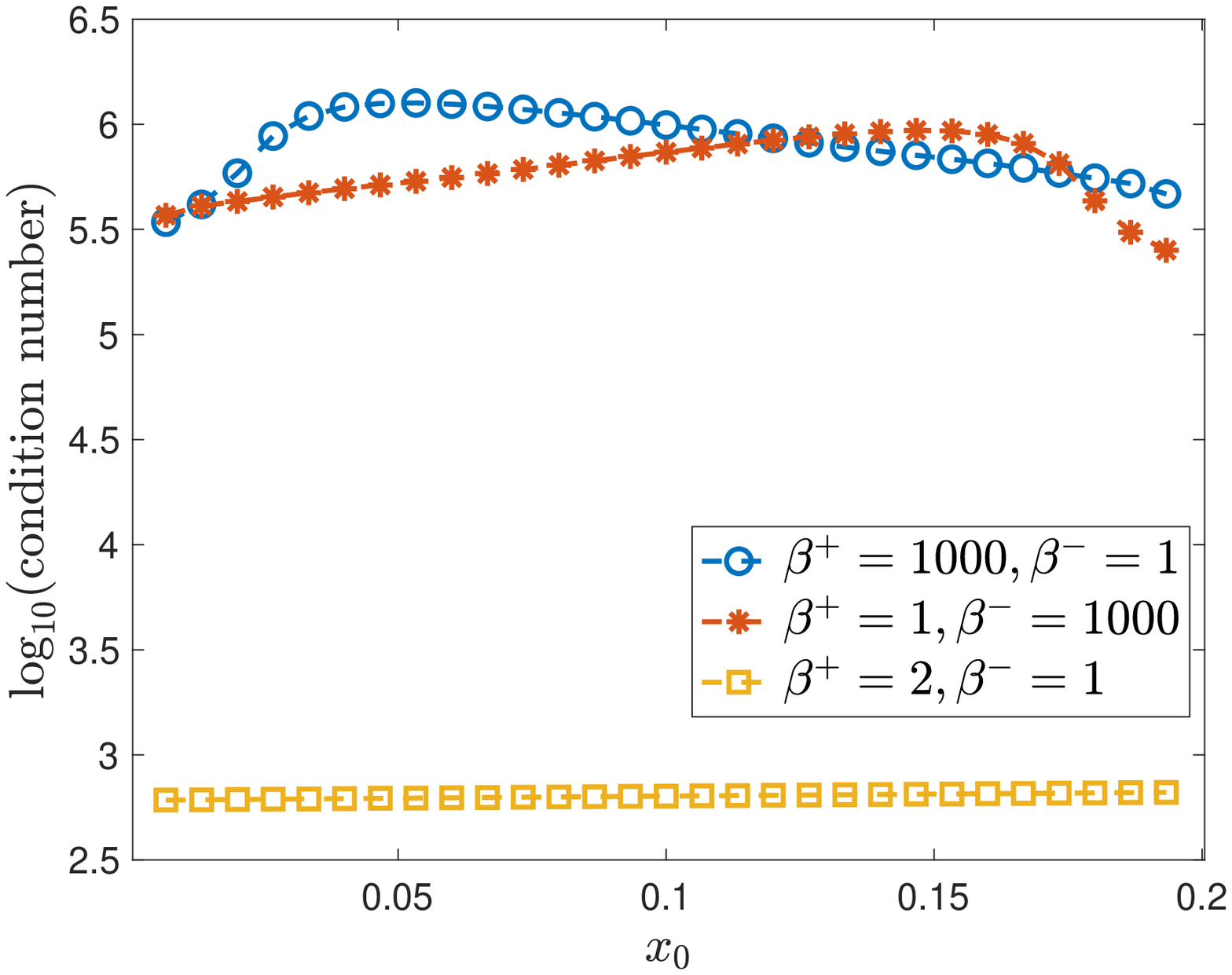}
\includegraphics[width=0.45\textwidth]{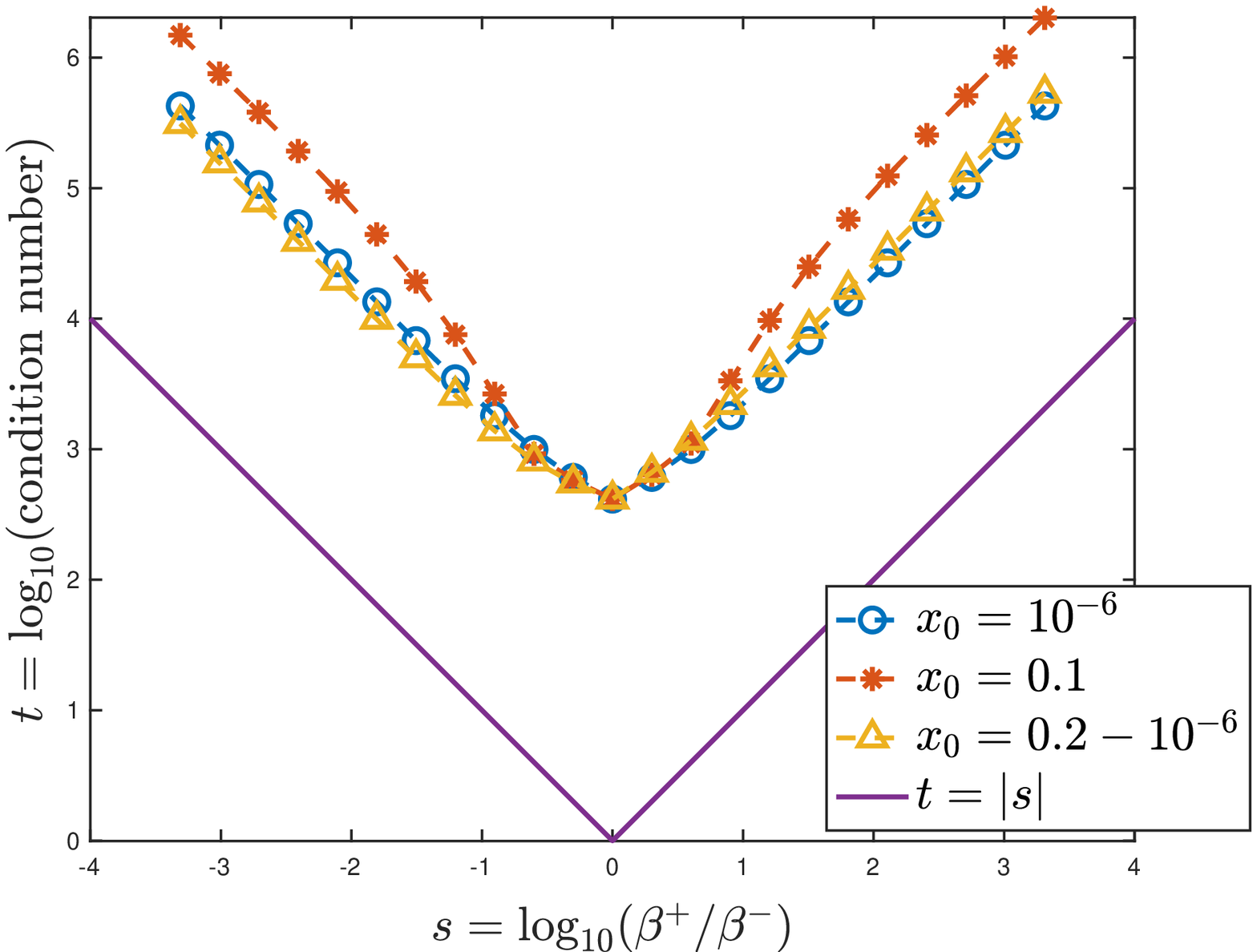}
 \caption{The dependence of the condition numbers on small-cut elements and the contrast $\beta^+/\beta^-$. The domain $(-1,1)^3$ is divided into $10\times 10\times 10$ cubes, and then each of these cubes is divided into six tetrahedrons. The interface is the plane $x=x_0$. It is easy to see that small-cut elements appear as $x_0\rightarrow 0$ or $x_0\rightarrow 0.2$.\label{Fig_ex3}} %% label for entire figure
\end{figure}

\section{Concluding remarks}\label{sec_con}
In this paper we have developed and analyzed an immersed Crouzeix-Raviart finite element method for solving 2D and 3D elliptic interface problems with scalar- and tensor-valued coefficients on unfitted meshes. We have shown that the IFE basis functions are unisolvent on arbitrary triangles/tetrahedrons cut by arbitrary interfaces and the IFE space has optimal approximation capabilities for the functions satisfying the interface conditions. The proposed method is easy to implement because that the curved interface is approximated by a continuous piecewise linear function via discrete level set functions and the coefficient is also approximated according to the discrete interface. We provide a complete error analysis of the proposed method taking into account all aspects of the approximation. The condition number the stiffness matrix of the proposed method is also proved to have the usual bound as that of conventional finite element methods. Throughout the analysis, the involved constants are independent of the mesh size and the interface position relative to the mesh.

\bibliographystyle{plain}
\bibliography{refer}
\end{document}